\documentclass[english,british]{article}
\usepackage[T1]{fontenc}
\usepackage[latin9]{inputenc}
\usepackage{geometry}
\usepackage{rotfloat}
\usepackage{enumitem}
\usepackage{amsthm}
\usepackage{amsmath}
\usepackage{amssymb}

\makeatletter

\providecommand{\tabularnewline}{\\}

\usepackage{enumitem}		
  \theoremstyle{remark}
  \newtheorem*{note*}{\protect\notename}
  \theoremstyle{remark}
  \newtheorem*{notation*}{\protect\notationname}
\theoremstyle{plain}
\newtheorem{thm}{\protect\theoremname}[section]
  \theoremstyle{plain}
  \newtheorem{prop}[thm]{\protect\propositionname}
  \theoremstyle{definition}
  \newtheorem{defn}[thm]{\protect\definitionname}
  \theoremstyle{plain}
  \newtheorem{lem}[thm]{\protect\lemmaname}
  \theoremstyle{definition}
  \newtheorem*{example*}{\protect\examplename}
 \newlist{casenv}{enumerate}{4}
 \setlist[casenv]{leftmargin=*,align=left,widest={iiii}}
 \setlist[casenv,1]{label={{\itshape\ \casename} \arabic*.},ref=\arabic*}
 \setlist[casenv,2]{label={{\itshape\ \casename} \roman*.},ref=\roman*}
 \setlist[casenv,3]{label={{\itshape\ \casename\ \alph*.}},ref=\alph*}
 \setlist[casenv,4]{label={{\itshape\ \casename} \arabic*.},ref=\arabic*}
  \theoremstyle{plain}
  \newtheorem{cor}[thm]{\protect\corollaryname}

\newcommand{\xyR}[1]{
 \xydef@\xymatrixrowsep@{#1}}
\newcommand{\xyC}[1]{
  \xydef@\xymatrixcolsep@{#1}}
\date{}
    \DeclareFontFamily{U}{wncy}{}
    \DeclareFontShape{U}{wncy}{m}{n}{<->wncyr10}{}
    \DeclareSymbolFont{mcy}{U}{wncy}{m}{n}
    \DeclareMathSymbol{\Sh}{\mathord}{mcy}{"58} 

\makeatother

\usepackage{babel}
  \addto\captionsbritish{\renewcommand{\corollaryname}{Corollary}}
  \addto\captionsbritish{\renewcommand{\definitionname}{Definition}}
  \addto\captionsbritish{\renewcommand{\examplename}{Example}}
  \addto\captionsbritish{\renewcommand{\lemmaname}{Lemma}}
  \addto\captionsbritish{\renewcommand{\notationname}{Notation}}
  \addto\captionsbritish{\renewcommand{\notename}{Note}}
  \addto\captionsbritish{\renewcommand{\propositionname}{Proposition}}
  \addto\captionsbritish{\renewcommand{\theoremname}{Theorem}}
  \addto\captionsenglish{\renewcommand{\corollaryname}{Corollary}}
  \addto\captionsenglish{\renewcommand{\definitionname}{Definition}}
  \addto\captionsenglish{\renewcommand{\examplename}{Example}}
  \addto\captionsenglish{\renewcommand{\lemmaname}{Lemma}}
  \addto\captionsenglish{\renewcommand{\notationname}{Notation}}
  \addto\captionsenglish{\renewcommand{\notename}{Note}}
  \addto\captionsenglish{\renewcommand{\propositionname}{Proposition}}
  \addto\captionsenglish{\renewcommand{\theoremname}{Theorem}}
  \providecommand{\corollaryname}{Corollary}
  \providecommand{\definitionname}{Definition}
  \providecommand{\examplename}{Example}
  \providecommand{\lemmaname}{Lemma}
  \providecommand{\notationname}{Notation}
  \providecommand{\notename}{Note}
  \providecommand{\propositionname}{Proposition}
 \addto\captionsbritish{\renewcommand{\casename}{Case}}
 \addto\captionsenglish{\renewcommand{\casename}{Case}}
 \providecommand{\casename}{Case}
\providecommand{\theoremname}{Theorem}

\begin{document}
\selectlanguage{english}%
\global\long\def\zz{\mathbb{Z}}
\global\long\def\im{\operatorname{im}}
\global\long\def\re{\operatorname{re}}
\global\long\def\rr{\mathbb{R}}
\global\long\def\cc{\mathbb{C}}
\global\long\def\vv{\mathbb{V}}
\global\long\def\SL{\operatorname{SL}}
\global\long\def\ord{\operatorname{ord}}
\global\long\def\Spec{\operatorname{Spec}}
\global\long\def\qq{\mathbb{Q}}
\global\long\def\pp{\mathbb{P}}
\global\long\def\Hom{\operatorname{Hom}}
\global\long\def\id{\operatorname{id}}
\global\long\def\gcd{\operatorname{gcd}}
\global\long\def\lcm{\operatorname{lcm}}
\global\long\def\Sym{\operatorname{Sym}}
\global\long\def\ff{\mathbb{F}}
\global\long\def\nn{\mathbb{N}}
\global\long\def\hh{\mathbb{H}}
\global\long\def\Ann{\operatorname{Ann}}
\global\long\def\End{\operatorname{End}}
\global\long\def\sgn{\operatorname{sgn}}
 \global\long\def\rk{\operatorname{rk}}
\global\long\def\Br{\operatorname{Br}}
 \global\long\def\Hasse{\operatorname{Hasse}}
\global\long\def\GL{\operatorname{GL}}
\global\long\def\Coll{\operatorname{Col}}
 \global\long\def\Tr{\operatorname{Tr}}
\global\long\def\aa{\mathbb{A}}
\global\long\def\ld{\operatorname{in}_{<}}
\global\long\def\ob{\operatorname{Ob}}
\global\long\def\mor{\operatorname{mor}}
\global\long\def\ext{\operatorname{Ext}}
\global\long\def\tor{\operatorname{Tor}}
\global\long\def\cok{\operatorname{coker}}
\global\long\def\ilim{\varprojlim}
\global\long\def\Gal{\operatorname{Gal}}

\title{On the Monodromy and Galois Group of Conics Lying on Heisenberg Invariant
Quartic K3 Surfaces}

\author{Florian Bouyer}
\maketitle
\selectlanguage{british}%
\begin{abstract}
In \cite{eklund2010}, Eklund showed that a general $(\zz/2\zz)^{4}$-invariant
quartic K3 surface contains at least $320$ conics. In this paper
we analyse the field of definition of those conics as well as their
Monodromy group. As a result, we prove that the moduli space $(\zz/2\zz)^{4}$-invariant
quartic K3 surface with a marked conic has $10$ irreducible components. 
\end{abstract}
\selectlanguage{english}%

\section{Introduction}

\selectlanguage{british}%
Consider the $(\zz/2\zz)^{4}$ subgroup of $\mathrm{Aut}(\pp_{\overline{\qq}}^{3})$
generated by the four transformations
\[
[x:y:z:w]\mapsto[y:x:w:z],[z:w:x:y],[x:y:-z:-w],[x:-y:z:-w].
\]
 The family of all quartic surfaces in $\pp_{\overline{\qq}}^{3}$
which are invariant under these transformations is known to be parameterised
by $\pp^{4}$ and has been studied extensively by \cite{barth1994,eklund2010}.
In his paper, \cite{eklund2010}, Eklund shows that a general such
quartic surface contains at least $320$ conics. This paper is lead
by the question, if such a surface is defined over a number field
$K$, what is the smallest extension for the conics on it to be defined?
Upon answering this question we look into the Monodromy group of the
conics and link it back to the Galois group of the field extension.

As a result of this paper, we deduce that the moduli space of pairs
$(X,C)$, where $X$ is an Heisenberg-invariant quartic K3 surfaces
and $C$ a conic on $X$, has 10 irreducible components. We contrast
this with universal Severi varieties of K3 surfaces. Let $X$ be a
primitive K3 surface of genus $g$, and $L$ an ample primitive line
bundle on $X$ such that $L^{2}=2g-2$ (so $g>1$). A Severi variety
of $X$ is $V_{k,h}(X)=\{C\in\left|kL\right|:C\mathrm{\, is\, irreducible\, nodal\, and\,}g(C)=h\}$.
The universal variety $\mathcal{V}_{k,h}^{g}$ can be considered as
the moduli space of pairs $(X,C)$, where $X$ is a primitive K3 surface
of genus $g$, and $C\in V_{k,h}(X)$. It is conjecture that all universal
Severi varieties $\mathcal{V}_{k,h}^{g}$ are irreducible. This conjecture
has been proven by Ciliberto and Dedieu for $3\leq g\leq11$, $g\neq10$
in \cite{MR2945592}. Kemeny showed that $\mathcal{\overline{V}}_{0,1}^{g}$
is connected for $g>2$ in \cite{MR3033964}, if it can be proven
that $\overline{\mathcal{V}}_{0,1}^{g}$ is smooth then it is irreducible. 

In Section \ref{sec:Background} we set the notations and review known
results that we will use in the rest of the paper. In Section \ref{sec:The-Galois-Group}
we start with the K3 surface parameterised by a point $p\in\pp^{4}$
and find the explicit equations, in terms of $p$, of the lines and
conics lying said K3 surface. We use this to construct the field extension
and show that, in general, the Galois group of the field extension
is $C_{2}^{10}$. We then look at the Monodromy group of the conics
in Section \ref{sec:Monodromy-Group}. In particular, we find generators
of the group and how it acts on the $320$ conics as well as showing
it is isomorphic to the Galois group. 
\begin{note*}
Throughout this paper, most calculations and equation rearrangements
were done using the computer algebra package Magma \cite[V2.21-4 ]{MR1484478}.
\end{note*}

\section{Background\label{sec:Background}}

We will be studying the following scheme
\[
\mathcal{X}:=\{A(x^{4}+y^{4}+z^{4}+w^{4})+Bxyzw+C(x^{2}y^{2}+z^{2}w^{2})+D(x^{2}z^{2}+y^{2}w^{2})+E(x^{2}w^{2}+y^{2}z^{2})=0\}\subset\pp_{[x:y:z:w]}^{3}\times\pp_{[A:B:C:D:E]}^{4}
\]
defined over $\overline{\qq}$. This can be viewed as a family of
quartic surfaces over $\pp^{3}$ parameterised by points $[A,B,C,D,E]$
in $\pp^{4}$.
\begin{notation*}
We will use $X_{p}$ and $[A,B,C,D,E]$ to denote the surface parametrised
by the point $p=[A,B,C,D,E]\in\pp^{4}$. \end{notation*}
\begin{note*}
In the case the surface $X_{p}$ is smooth, then it is a K3 surface.
\end{note*}
Consider the group $\Omega$ acting on $\pp^{3}\times\pp^{4}$ generated
by the following five elements
\[
[x,y,z,w,A,B,C,D,E]\mapsto\begin{cases}
[x,y,z,-w,A,-B,C,D,E] & \phi_{1}\\
{}[x,y,w,z,A,B,C,E,D] & \phi_{2}\\
{}[x,z,y,w,A,B,D,C,E] & \phi_{3}\\
{}[x,y,iz,iw,A,-B,C,-D,-E] & \phi_{4}\\
{}[x-y,x+y,z-w,z+w,2A+C,8(D-E),12A-2C,B+2D+2E,-B+2D+2E] & \phi_{5}
\end{cases}.
\]
This group fixes $\mathcal{X}$ and hence is a subgroup of its automorphism
group. While this group is rather large with order $2^{4}\cdot6!$,
we can pick out the following normal subgroup generated by 
\begin{itemize}
\item $\gamma_{1}:=\phi_{3}\phi_{4}^{2}\phi_{3}\phi_{5}^{2}$,
\item $\gamma_{2}:=\phi_{4}^{2}\phi_{3}\phi_{5}^{2}\phi_{3}$,
\item $\gamma_{3}:=\phi_{4}^{2}$\label{pg:gamma_3},
\item $\gamma_{4}:=\phi_{3}\phi_{4}^{2}\phi_{3}$,
\end{itemize}
which we denote by $\Gamma$. The group $\Gamma$ consists of all
elements of $\Omega$ which fix $\pp_{[A:B:C:D:E]}^{4}$ in $\pp^{3}\times\pp^{4}$.
In particular upon picking a point $p\in\pp^{4}$ we have that $\Gamma$
is a subgroup of $\mathrm{Aut}(X_{p})$ (when projecting the elements
of $\Gamma$ onto the $\pp_{[x:y:z:w]}^{3}$ component). Explicitly,
when regarding $\Gamma$ as acting on $\pp^{3}$, we have that its
generators are
\[
[x,y,z,w]\mapsto\begin{cases}
[y,x,w,z] & \gamma_{1}\\
{}[z,w,x,y] & \gamma_{2}\\
{}[x,y,-z,-w] & \gamma_{3}\\
{}[x,-y,z,-w] & \gamma_{4}
\end{cases}.
\]
From this we know that $\Gamma\cong C_{2}^{4}$. We also note that
$\Omega/\Gamma\cong S_{6}$ (but $\Omega\ncong C_{2}^{4}\times S_{6}$
in particular because $\Omega$ has trivial centre). \foreignlanguage{english}{We
also make a note of the points of $\pp^{3}$ which are fixed by $\gamma\in\Gamma\setminus\{\id\}$
(restricted to $\pp^{3}$). Each such $\gamma$ has two skew line
$L,\overline{L}$ of fixed points which are given by its $(+1)$ and
$(-1)$ eigenspace, respectively its $(+i)$ and $(-i)$ eigenspace.}
The lines are given in Table~\ref{tab:List-of-Invariant} (more information
is given which will be explained after Proposition \ref{prop:sing pt from sufracce}).
Note that every pair of lines is also fixed by any $\gamma\in\Gamma$
(on top of containing the fixed points of a particular $\gamma$).

\selectlanguage{english}%
\begin{table}[h]
\begin{centering}
\begin{tabular}{|c|l|l|l|}
\cline{2-4} 
\multicolumn{1}{c|}{\selectlanguage{british}%
\selectlanguage{english}%
} & $L_{i}$ & $\overline{L_{i}}$ & \selectlanguage{british}%
Segre Plane\selectlanguage{english}%
\tabularnewline
\hline 
$\gamma_{1}$ & $[s:s:t:t]$ & $[s:-s:t:-t]$ & \selectlanguage{british}%
$q_{+C}=p_{+0}=p_{-1}=0$\selectlanguage{english}%
\tabularnewline
\hline 
$\gamma_{2}$ & $[s:t:s:t]$ & $[s:t:-s:-t]$ & \selectlanguage{british}%
$q_{+D}=p_{+0}=p_{-2}=0$\selectlanguage{english}%
\tabularnewline
\hline 
\selectlanguage{british}%
$\gamma_{1}\gamma_{2}$\selectlanguage{english}%
 & $[s:t:t:s]$ & $[s:t:-t:-s]$ & \selectlanguage{british}%
$q_{+E}=p_{+0}=p_{-3}=0$\selectlanguage{english}%
\tabularnewline
\hline 
\selectlanguage{british}%
$\gamma_{3}$\selectlanguage{english}%
 & \selectlanguage{british}%
$[s:t:0:0]$\selectlanguage{english}%
 & \selectlanguage{british}%
$[0:0:s:t]$\selectlanguage{english}%
 & \selectlanguage{british}%
$A=q_{+C}=q_{-C}=0$\selectlanguage{english}%
\tabularnewline
\hline 
\selectlanguage{british}%
$\gamma_{1}\gamma_{3}$\selectlanguage{english}%
 & \selectlanguage{british}%
$[s:-s:t:t]$\selectlanguage{english}%
 & \selectlanguage{british}%
$[s:s:t:-t]$\selectlanguage{english}%
 & \selectlanguage{british}%
$q_{-C}=p_{-0}=p_{+1}=0$\selectlanguage{english}%
\tabularnewline
\hline 
\selectlanguage{british}%
$\gamma_{2}\gamma_{3}$\selectlanguage{english}%
 & \selectlanguage{british}%
$[s:t:is:it]$\selectlanguage{english}%
 & \selectlanguage{british}%
$[s:t:-is:-it]$\selectlanguage{english}%
 & \selectlanguage{british}%
$q_{-D}=p_{-1}=p_{+3}=0$\selectlanguage{english}%
\tabularnewline
\hline 
\selectlanguage{british}%
$\gamma_{1}\gamma_{2}\gamma_{3}$\selectlanguage{english}%
 & \selectlanguage{british}%
$[s:t:it:is]$\selectlanguage{english}%
 & \selectlanguage{british}%
$[s:t:-it:-is]$\selectlanguage{english}%
 & \selectlanguage{british}%
$q_{-E}=p_{-1}=p_{+2}=0$\selectlanguage{english}%
\tabularnewline
\hline 
\selectlanguage{british}%
$\gamma_{4}$\selectlanguage{english}%
 & \selectlanguage{british}%
$[s:0:t:0]$\selectlanguage{english}%
 & \selectlanguage{british}%
$[0:s:0:t]$\selectlanguage{english}%
 & \selectlanguage{british}%
$A=q_{+D}=q_{-D}=0$\selectlanguage{english}%
\tabularnewline
\hline 
\selectlanguage{british}%
$\gamma_{1}\gamma_{4}$\selectlanguage{english}%
 & \selectlanguage{british}%
$[s:is:t:it]$\selectlanguage{english}%
 & \selectlanguage{british}%
$[s:-is:t:-it]$\selectlanguage{english}%
 & \selectlanguage{british}%
$q_{-C}=p_{-2}=p_{+3}=0$\selectlanguage{english}%
\tabularnewline
\hline 
\selectlanguage{british}%
$\gamma_{2}\gamma_{4}$\selectlanguage{english}%
 & \selectlanguage{british}%
$[s:t:-s:t]$\selectlanguage{english}%
 & \selectlanguage{british}%
$[s:t:s:-t]$\selectlanguage{english}%
 & \selectlanguage{british}%
$q_{+D}=p_{-0}=p_{+2}=0$\selectlanguage{english}%
\tabularnewline
\hline 
\selectlanguage{british}%
$\gamma_{1}\gamma_{2}\gamma_{4}$\selectlanguage{english}%
 & \selectlanguage{british}%
$[s:t:it:-is]$\selectlanguage{english}%
 & \selectlanguage{british}%
$[s:t:-it:is]$\selectlanguage{english}%
 & \selectlanguage{british}%
$q_{-E}=p_{+1}=p_{-2}=0$\selectlanguage{english}%
\tabularnewline
\hline 
\selectlanguage{british}%
$\gamma_{3}\gamma_{4}$\selectlanguage{english}%
 & \selectlanguage{british}%
$[s:0:0:t]$\selectlanguage{english}%
 & \selectlanguage{british}%
$[0:s:t:0]$\selectlanguage{english}%
 & \selectlanguage{british}%
$A=q_{+E}=q_{-E}=0$\selectlanguage{english}%
\tabularnewline
\hline 
\selectlanguage{british}%
$\gamma_{1}\gamma_{3}\gamma_{4}$\selectlanguage{english}%
 & \selectlanguage{british}%
$[s:-is:t:it]$\selectlanguage{english}%
 & \selectlanguage{british}%
$[s:is:t:-it]$\selectlanguage{english}%
 & \selectlanguage{british}%
$q_{-C}=p_{+2}=p_{-3}=0$\selectlanguage{english}%
\tabularnewline
\hline 
\selectlanguage{british}%
$\gamma_{2}\gamma_{3}\gamma_{4}$\selectlanguage{english}%
 & \selectlanguage{british}%
$[s:t:-is:it]$\selectlanguage{english}%
 & \selectlanguage{british}%
$[s:t:is:-it]$\selectlanguage{english}%
 & \selectlanguage{british}%
$q_{-D}=p_{+1}=p_{-3}=0$\selectlanguage{english}%
\tabularnewline
\hline 
\selectlanguage{british}%
$\gamma_{1}\gamma_{2}\gamma_{3}\gamma_{4}$\selectlanguage{english}%
 & \selectlanguage{british}%
$[s:t:t:-s]$\selectlanguage{english}%
 & \selectlanguage{british}%
$[s:t:-t:s]$\selectlanguage{english}%
 & \selectlanguage{british}%
$q_{+W}=p_{-0}=p_{+3}=0$\selectlanguage{english}%
\tabularnewline
\hline 
\end{tabular}
\par\end{centering}

\caption{\label{tab:List-of-Invariant}List of Invariant Lines}
\end{table}

\selectlanguage{british}%
\begin{notation*}
We shall denote by $\mathcal{L}$ the union of the $15$ pairs of
lines.
\end{notation*}
We now consider the cases when $X_{p}$ is not a smooth surface using
the following proposition taken from \cite[Prop 2.1]{eklund2010}.
\selectlanguage{english}%
\begin{prop}
\label{prop:Singular Conditions}Let $p=[A,B,C,D,E]\in\pp^{4}$. The
surface $X_{p}$ is singular if and only if
\[
A\cdot\Delta\cdot q_{+C}\cdot q_{-C}\cdot q_{+D}\cdot q_{-D}\cdot q_{+E}\cdot q_{-E}\cdot p_{+0}\cdot p_{+1}\cdot p_{+2}\cdot p_{+3}\cdot p_{-0}\cdot p_{-1}\cdot p_{-2}\cdot p_{-3}=0,
\]
where:
\begin{equation}
\Delta=16A^{3}+AB^{2}-4A(C^{2}+D^{2}+E^{2})+4CDE\label{eq:Delta}
\end{equation}
\begin{flalign*}
q_{+C}= & 2A+C & q_{+D}= & 2A+D & q_{+E}= & 2A+E\\
q_{-C}= & 2A-C & q_{-D}= & 2A-D & q_{-E}= & 2A-E\\
p_{+0}= & 4A+B+2C+2D+2E &  &  & p_{-0}= & 4A-B+2C+2D+2E\\
p_{+1}= & 4A+B+2C-2D-2E &  &  & p_{-1}= & 4A-B+2C-2D-2E\\
p_{+2}= & 4A+B-2C+2D-2E &  &  & p_{-2}= & 4A-B-2C+2D-2E\\
p_{+3}= & 4A+B-2C-2D+2E &  &  & p_{-3}= & 4A-B-2C-2D+2E.
\end{flalign*}
\end{prop}
\begin{defn}
The surface $S_{3}=\left\{ \Delta=0\right\} \subset\pp^{4}$ is the
\textit{Segre cubic}\textit{\emph{. The $15$ hyperplanes in $\pp^{4}$
defined by the $15$ equations above ($q_{+C},\dots,p_{-3}$) shall
be referred as the }}\textit{singular hyperplanes}\textit{\emph{. }}
\end{defn}
\selectlanguage{british}%
The Segre cubic has $10$ nodes, namely:\foreignlanguage{english}{
\[
[1,0,-2,-2,2],[1,0,-2,2,-2],[1,0,2,-2,-2],[1,0,2,2,2],
\]
\[
[0,-2,1,0,0],[0,2,1,0,0],[0,-2,0,1,0],[0,2,0,1,0],[0,-2,0,0,1],\mathrm{\, and\,}[0,2,0,0,1].
\]
We shall denote these $10$ points by $q_{i}$, $i\in[1,\dots,10]$,
as ordered above. These $10$ points have associated quartic in $\pp^{3}$,
which turns out to be quadrics of multiplicity $2$. We denote the
quadric associated to the point $q_{i}$ by $Q_{i}$. Explicitly they
are:
\[
x^{2}-y^{2}-z^{2}+w^{2}=0\,,\, x^{2}-y^{2}+z^{2}-w^{2}=0\,,\, x^{2}+y^{2}-z^{2}-w^{2}=0\,,\, x^{2}+y^{2}+z^{2}+w^{2}=0\,,
\]
\[
xy-zw=0\,,\, xy+zw=0\,,\, xz-yw=0\,,\, xz+yw=0\,,\, xw-yz=0\,,\,\mathrm{\, and\,}xw+yz=0.
\]
 }

\selectlanguage{english}%
It is known that for a general point on $S_{3}$ the corresponding
surface is a Kummer surface (\cite[Prop 2.2]{eklund2010}). The following
proposition links such Kummer surfaces with their singular points.
\begin{prop}
\label{prop:Surface from sing pt}Let $p=[x,y,z,w]$ be a point in
$\pp^{3}\setminus\mathcal{L}$ and let 
\begin{itemize}
\item $A=(yz+xw)(yz-xw)(xz+yw)(xz-yw)(zw+xy)(zw-xy)$,
\item {\small{$B=2xyzw(-x^{2}-y^{2}+z^{2}+w^{2})(-x^{2}+y^{2}+z^{2}-w^{2})(x^{2}-y^{2}+z^{2}-w^{2})(x^{2}+y^{2}+z^{2}+w^{2})$,}}{\small \par}
\item $C=(yz+xw)(yz-xw)(xz+yw)(xz-yw)(x^{4}+y^{4}-z^{4}-w^{4})$,
\item $D=(yz+xw)(yz-xw)(zw+xy)(zw-xy)(-x^{4}+y^{4}-z^{4}+w^{4})$,
\item $E=(xz+yw)(xz-yw)(zw+xy)(zw-xy)(x^{4}-y^{4}-z^{4}+w^{4}).$
\end{itemize}
\selectlanguage{british}%
Then the point $[A,B,C,D,E]$ lies on the Segre cubic and the associated
Kummer surface has the 16 singular points $\{\gamma([x,y,z,w]):\gamma\in\Gamma\}$. \end{prop}
\selectlanguage{british}%
\begin{proof}
Let $F=A(X^{4}+Y^{4}+Z^{4}+W^{4})+\dots+E(X^{2}W^{2}+Z^{2}Y^{2})$.
By algebraic manipulation, the system of linear equations
\[
\frac{\partial F}{\partial X}(p)=\frac{\partial F}{\partial Y}(p)=\frac{\partial F}{\partial Z}(p)=\frac{\partial F}{\partial W}(p)=0
\]
has a unique solution $[A,B,C,D,E]\in\pp^{4}$ as given above. Note
that in particular, such a point $p$ defines uniquely the Kummer
surface of which it is a singular point.\end{proof}
\begin{prop}
\label{prop:sing pt from sufracce}Let $[A,B,C,D,E]$ be a point on
the Segre cubic not lying on one of the $15$ singular hyperplanes.
Then the associated surface's 16 singular points are $[x,y,z,w]$
where $x,y,z$ and $w$ solve the following equations
\begin{itemize}
\item $az^{8}+bz^{6}w^{2}+cz^{4}w^{4}+bz^{2}w^{6}+aw^{8}=0$, where $a=-A^{2}B^{2}$,
$b=4(2AD-CE)(2AE-CD)$ and $c=2(A^{2}B^{2}-2(E^{2}+D^{2})(4A^{2}+C^{2})+16ACDE)$,
\item $(4A^{2}-C^{2})(Ez^{2}-Dw^{2})y^{2}+A((4A^{2}-C^{2})(z^{4}-w^{4})+(E^{2}-D^{2})(z^{4}+w^{4}))+C(E^{2}-D^{2})z^{2}w^{2}=0$,
\selectlanguage{english}%
\item \textup{$2(C^{2}-4A^{2})xyzw+BCz^{2}w^{2}+ABw^{4}+ABz^{4}=0.$}
\end{itemize}
\end{prop}
\begin{proof}
Without loss of generality, we assume $w=1$. Then the first equation
can be considered as a symmetric quartic polynomial with the variable
$z^{2}$, and hence $z$ can be written as a radical function of $A,B,C,D,E$,
i.e,
\[
z=\pm\sqrt{\frac{u_{\pm}\pm\sqrt{u_{\pm}^{2}-4}}{2}},\mbox{ where }u_{\pm}=\frac{-b\pm\sqrt{b^{2}-4a(c-2a)}}{2a}.
\]
 Similarly, we can write $x$ and $y$ as radical functions of $A,B,C,D,E$.
Substituting the point $[x,y,z,1]$ (written in terms of $A,B,C,D,E$)
into the equations of Proposition \ref{prop:Surface from sing pt},
we get an equality. Since a point defines the Kummer surface uniquely,
we must have that the point $[x,y,z,1]$ is a singular point to $[A,B,C,D,E]$.
\end{proof}
We explain why we need the hypothesis in the two previous propositions,
namely taking a point in $\pp^{3}$ away from $\mathcal{L}$ and taking
a point in $\pp^{4}$ away from the singular hyperplanes. Fist we
note that the intersection of one of the singular hyperplanes with
the Segre cubic breaks down into $3$ planes. For example 
\begin{eqnarray*}
\{q_{+C}=0\}\cap\{\Delta=0\} & = & \{q_{+C}=0,q_{-C}=0,A=0\}\\
 & \cup & \{q_{+C}=0,p_{-0}=0,p_{+1}=0\}\\
 & \cup & \{q_{+C}=0,p_{+0}=0,p_{-1}=0\}.
\end{eqnarray*}
 One can check that we get $15$ planes this way, which we shall refer
to as the\emph{ $15$ Segre planes}. 

Suppose the surface $X_{p}$ is represented by a point $p$ lying
on one of the 15 Segre planes, that is it doesn't satisfy the hypothesis
of Proposition \ref{prop:sing pt from sufracce}. Then, one can calculate,
that $X_{p}$ does not have only $16$ singular points, but rather
two skew singular lines. Namely one of the $15$ pairs of lines in
$\mathcal{L}$. On the other hand, consider the surface $X_{p}$,
with $p\in\pp^{4}$, which has the singular point $q\in\mathcal{L}$.
By Proposition \ref{prop:Singular Conditions} we know that either
$p$ lies on the Segre cubic or on one of the $15$ singular hyperplanes.
If $p$ lies on the Segre cubic, then in fact $p$ lies on one of
the Segre planes. If $p$ lies on a singular hyperplane and not on
the Segre cubic, then $q$ lies on $3$ lines contained in $\mathcal{L}$. 

Hence we have a one to one correspondence between the $15$ pairs
of skew lines of $\mathcal{L}$ and the $15$ Segre planes. Table
\ref{tab:List-of-Invariant} shows which Segre plane corresponds to
which pair of lines.
\selectlanguage{english}%
\begin{defn}
Let $Y$ be a quartic surface in $\pp^{3}$. We say that a plane $T$
in $\pp^{3}$ is a \emph{trope} of $Y$ if $Y\cap T$ is an irreducible
conic counted with multiplicity two. \end{defn}
\begin{lem}
\label{lem:Trope iff sing}A quartic surface $Y\subset\pp^{3}$ which
has a trope $T$ is necessarily singular.
\end{lem}
We now turn to the theorem from Eklund, \cite[Thm 4.3]{eklund2010},
which started the idea of this paper.
\begin{thm}
\label{thm:320-Conics}A general K3 surface $X$ from the family $\mathcal{X}$
contains at least $320$ smooth conics. \end{thm}
\begin{proof}
The proof is adapted from \cite[Thm 4.3]{eklund2010}, which we reproduce
here as we will use some elements of the proof in the rest of this
paper. For this proof if $Y$ is a hypersurface, fix $\widetilde{Y}$
to be an equation defining $Y$.

Pick $p\in\pp^{4}$ general and let $q_{i}$ be a node of the Segre
cubic $S_{3}$ (in particular fix $i$). We have that the associated
surface to $q_{i}$ is $Q_{i}^{2}$, a quadric of multiplicity two.
The line through $p$ and $q_{i}$ intersects $S_{3}$ in exactly
one more point, call it $p_{i}$. Hence we have $\widetilde{X_{p}}=\alpha\left(\widetilde{Q_{i}}\right)^{2}+\alpha'\widetilde{X_{p_{i}}}$
for some $\alpha,\alpha'\in\overline{\qq}$. As $p$ is general by
Proposition \ref{prop:sing pt from sufracce}, we have that the associated
surface $X_{p_{i}}$ is Kummer. Pick a singular point on $X_{p_{i}}$
and consider its dual $T$. As $T$ is a trope of $X_{p_{i}}$ (see
\cite[pg 12]{eklund2010} for more details) we have that $\widetilde{X_{p_{i}}}=\mu(Q')^{2}+\lambda\widetilde{T}$,
for some $\mu\in\overline{\qq}$, a cubic equation $\lambda$ and
a quadratic equation $Q'$. Hence, as equations, 
\begin{eqnarray}
\widetilde{X_{p}} & = & \alpha\left(\widetilde{Q_{i}}\right)+\alpha'\mu\left(Q'\right)^{2}+\alpha'\lambda\widetilde{T}\nonumber \\
 & = & (\sqrt{\alpha}\widetilde{Q_{i}}+\sqrt{-\alpha'\mu}Q')(\sqrt{\alpha}\widetilde{Q_{i}}-\sqrt{-\alpha'\mu}Q')+\alpha'\lambda\widetilde{T}.\label{eq:of X_p}
\end{eqnarray}
So $X_{p}\cap T$ is the union of two conics. As the general member
of the family does not contain any lines (see \cite[Prop 2.3]{eklund2010}),
nor does it have a trope (Lemma \ref{lem:Trope iff sing}), we have
that the two conics are smooth and distinct.

Since a general Kummer surface of $\mathcal{X}$ is determined by
any of its tropes (i.e., by its singular points Proposition \ref{prop:Surface from sing pt}),
all the tropes defined by using the $10$ nodes $q$ of $S_{3}$ are
different. As two different planes cannot have a smooth conic in common,
we conclude that we have constructed $10\cdot16\cdot2=320$ smooth
conics on $X_{p}$.
\end{proof}
\selectlanguage{british}%

\section{The Galois Group\label{sec:The-Galois-Group}}

In this section we are going to shift away from working over $\overline{\qq}$
to working over number fields. Let $K$ be a number field, and let
$p$ be a general point in $\pp_{K}^{4}$. Then the associated K3
surface, $X_{p}$, has $320$ conics on it, so let $L$ be the smallest
number field over which those conics are defined. Note that it must
be an extension of $K$. We want to work out the Galois group of the
field of definition of the $320$ conics. That is, we are interested
in $\mathrm{Gal}(L/K)$. To do this we will first find $L$. For this,
by Theorem \ref{thm:320-Conics}, it is sufficient to work out, for
each conic, over \foreignlanguage{english}{what field $\alpha,\alpha',\mu$
and $Q'$ are defined over (since $Q_{i}$ is defined over $\qq$
for all $i$). }

\selectlanguage{english}%
Let $p=[A,B,C,D,E]\in\pp^{4}$, we have that $\alpha$ and $\alpha'$
depend only on $Q_{i}$ (or more specifically on the point $q_{i}$),
while $\mu$ and $Q'$ depend on the $Q_{i}$ and the trope $T$ (of
which there are $16$ choices once $Q_{i}$ is fixed). Let $\alpha_{i}$
and $\alpha_{i}'$ be associated to $Q_{i}$ and we first look at
$\alpha_{i},\alpha'_{i}$. Using the equations defining the line through
the point $p$ and the point $q_{i}$, and the cubic equation defining
$S_{3}$ we can find the point $p_{i}$. Hence we write $[A_{i},B_{i},C_{i},D_{i},E_{i}]=X_{p,i}$
in terms of $A,B,C,D\mbox{ and }E$. Since we know $X_{q_{i}}$, we
can use simple algebra to calculate $\alpha_{i}$ and $\alpha_{i}'$.
We find that $\alpha_{i}=\Delta\beta_{i}^{-1}$ and $\alpha'_{i}=\begin{cases}
\beta_{i}^{-1} & i\in[1,\dots,4]\\
(4\beta_{i})^{-1} & i\in[5,\dots,10]
\end{cases}$ where
\begin{flalign*}
\beta_{1}= & 12A^{2}+\frac{1}{4}B^{2}+4A(C+D-E)-(C^{2}+D^{2}+E^{2})+2(CD-CE-DE)\\
\beta_{2}= & 12A^{2}+\frac{1}{4}B^{2}+4A(C-D+E)-(C^{2}+D^{2}+E^{2})+2(-CD+CE-DE)\\
\beta_{3}= & 12A^{2}+\frac{1}{4}B^{2}+4A(-C+D+E)-(C^{2}+D^{2}+E^{2})+2(-CD-CE+DE)\\
\beta_{4}= & 12A^{2}+\frac{1}{4}B^{2}-4A(C+D+E)+(C^{2}+D^{2}+E^{2})+2(CD+CE+DE)
\end{flalign*}
 
\begin{gather*}
\beta_{5}=-(AB+2AC-DE)\qquad\qquad\qquad\beta_{6}=AB-2AC+DE\\
\beta_{7}=-(AB+2AD-CE)\qquad\qquad\qquad\beta_{8}=AB-2AD+CE\\
\beta_{9}=-(AB+2AE-CD)\qquad\qquad\qquad\beta_{10}=AB-2AE+CD
\end{gather*}
and $\Delta$ is defined by Equation (\ref{eq:Delta}). In particular,
letting $\widetilde{Y}$ be the equation defining the hypersurface
$Y$ as in Theorem \ref{thm:320-Conics}, $\widetilde{X_{p}}=\beta_{i}^{-1}(\Delta\left(\widetilde{Q_{i}}\right)^{2}+\widetilde{X_{p_{i}}})$,
with the $\frac{1}{4}$ factor absorbed in the equation defining $X_{p_{i}}$
when needed. Hence, using Equation (\ref{eq:of X_p}) and by rescaling
with $\beta$, we have, for a fixed $Q_{i}$ and $T$,
\[
X_{p}\cap T=\left\{ \left(\widetilde{Q_{i}}+\sqrt{-\frac{\mu_{i}}{\Delta}}Q'\right)\left(\widetilde{Q_{i}}-\sqrt{-\frac{\mu_{i}}{\Delta}}Q'\right)=0,\widetilde{X_{p}}=0\right\} 
\]
In particular, the field of definition of the pair of conics defined
by $T$ only depend on $\sqrt{-\frac{\mu_{i}}{\Delta}}$ and $Q'$.
So let us fix an $i$, and set $X_{p_{i}}=[A_{i},B_{i},C_{i},D_{i},E_{i}]$
and let us fix $T$ by choosing the singular point $[r_{3,i},r_{2,i},r_{1,i},1]$
on $X_{p_{i}}$. That is, $T$ is defined by $r_{3,i}x+r_{2,i}y+r_{1,i}z+w=0$.
If we use the equations in Proposition \ref{prop:Surface from sing pt}
to rewrite the equation defining $X_{p_{i}}$ in terms of $r_{3,i},r_{2,i}$
and $r_{1,i}$, then substituting $w=-(r_{3,i}x+r_{2,i}y+r_{1,i}z)$
into $X_{p_{i}}$ we find that $\left(Q'\right)^{2}=(a_{0}x^{2}+a_{1}y^{2}+a_{2}z^{2}+a_{3}xy+a_{4}xz+a_{4}yz)^{2}$
\label{pg:Q2}where 
\selectlanguage{british}%
\begin{itemize}
\item $a_{0}=(r_{2}r_{3}-r_{1})\cdot(r_{2}r_{3}+r_{1})\cdot(r_{1}r_{3}-r_{2})\cdot(r_{1}r_{3}+r_{2})$,
\item $a_{1}=(r_{2}r_{3}-r_{1})\cdot(r_{2}r_{3}+r_{1})\cdot(r_{1}r_{2}-r_{3})\cdot(r_{1}r_{2}+r_{3})$,
\item $a_{2}=(r_{1}r_{3}-r_{2})\cdot(r_{1}r_{3}+r_{2})\cdot(r_{1}r_{2}-r_{3})\cdot(r_{1}r_{2}+r_{3})$,
\item $a_{3}=r_{3}\cdot r_{2}\cdot(2r_{1}^{2}r_{2}^{2}r_{3}^{2}-r_{1}^{4}-r_{2}^{4}-r_{3}^{4}+1)$,
\item $a_{4}=r_{3}\cdot r_{1}\cdot(2r_{1}^{2}r_{2}^{2}r_{3}^{2}-r_{1}^{4}-r_{2}^{4}-r_{3}^{4}+1)$,
\item $a_{5}=r_{2}\cdot r_{1}\cdot(2r_{1}^{2}r_{2}^{2}r_{3}^{2}-r_{1}^{4}-r_{2}^{4}-r_{3}^{4}+1)$.
\end{itemize}
Now as each trope $T$, and hence each associated $Q'$, of $X_{p_{i}}$
are defined by $\Gamma$ acting on the point \foreignlanguage{english}{$[r_{3,i},r_{2,i},r_{1,i},1]$
we have that the $16$ conics $Q'$ associated to $X_{p_{i}}$ are
defined over }field $K(r_{1,i},r_{2,i})$ (recall that $r_{3,i}$
is a $K$-linear combination of $r_{1,i},r_{2,i}$, see Proposition
\ref{prop:sing pt from sufracce}). 

Next to work out $\mu_{i}$ (depending on the singular point \foreignlanguage{english}{$[r_{3,i},r_{2,i},r_{1,i},1]$}),
we use the fact that (as equations) $X_{p_{i}}=\mu\left(Q'\right)^{2}+\lambda T$
and that $Q'$ has no $w$ terms, to find that
\[
\mu_{i}=(A_{i}r_{1}^{4}+C_{i}r_{1}^{2}+A_{i})\cdot a_{2}^{-2}.
\]
Again we see that the action of $\Gamma$ on \foreignlanguage{english}{$[r_{3,i},r_{2,i},r_{1,i},1]$}
will give us the other $15$ $\mu$'s. In particular, as the 15 other
singular points have $z$-coordinates $\pm r_{1,i},\pm\frac{1}{r_{1,i}},\pm\frac{r_{2,i}}{r_{3,i}},\pm\frac{r_{3,i}}{r_{2,i}}$,
there are $3$ other $\mu$, namely $\frac{1}{r_{1,i}^{2}}\mu$, $\overline{\mu}=(A_{i}\frac{r_{2,i}^{4}}{r_{3,i}^{4}}+C_{i}\frac{r_{2,i}^{2}}{r_{3,i}^{2}}+A_{i})\cdot\overline{a}_{2}^{-2}$
and $\frac{r_{3,i}^{2}}{r_{2,i}^{2}}\overline{\mu}$ (where $\overline{a_{2}}$
can be calculated, but will not be needed). Putting all of this together
we have proven the following.
\begin{prop}
\label{thm:Field of definition of 32 conics}Let $p=[A,B,C,D,E]\in\pp^{4}$
and fix $i\in[1,\dots,10]$ . Let $[A_{i},B_{i},C_{i},D_{i},E_{i}]=p_{i}\in\pp^{4}$
be the third point of intersection between the Segre cubic $S_{3}$,
and the line joining $q_{i}$ and $p$. Then the $32$ conics lying
on $X_{p}$ and associated to the point $q_{i}$ (as per the construction
in Theorem \ref{thm:320-Conics}) are defined over
\begin{equation}
K_{i}=K(r_{1,i},r_{2,i},r_{\mu,i},\overline{r_{\mu,i}})\label{eq:K_i}
\end{equation}
 where{\small{
\begin{eqnarray}
r_{1,i} & \mbox{is a root of} & ax^{8}+bx^{6}+cx^{4}+bx^{2}+a\label{eq:first_root}\\
r_{2,i} & \mbox{is a root of} & d(E_{i}r_{1,i}^{2}-D_{i})x^{2}+A_{i}\left(d\left(r_{1,i}^{4}-1\right)+e\left(r_{1,i}^{4}+1\right)\right)+C_{i}er_{1,i}^{2}\label{eq:second_root}\\
r_{\mu,i} & \mbox{is a root of} & x^{2}+\frac{1}{\Delta}(A_{i}r_{1,i}^{4}+C_{i}r_{1,i}^{2}+A_{i})\\
\overline{r_{\mu,i}} & \mbox{is a root of} & x^{2}+\frac{1}{\Delta}(A_{i}\overline{r}_{1,i}^{4}+C_{i}\overline{r}_{1,i}^{2}+A_{i})\label{eq:last_root}
\end{eqnarray}
}}with $\overline{r}_{1,i}=\frac{r_{2,i}}{r_{3,i}}$ (which can be
expressed in terms of $r_{1,i}$) and 
\[
a=-A_{i}^{2}B_{i}^{2}
\]
\[
b=4(2A_{i}D_{i}-C_{i}E_{i})(2A_{i}E_{i}-C_{i}D_{i})
\]
\[
c=2(A_{i}^{2}B_{i}^{2}-2(E_{i}^{2}+D_{i}^{2})(4A_{i}^{2}+C_{i}^{2})+16A_{i}C_{i}D_{i}E_{i})
\]
\begin{flalign*}
d & =4A_{i}^{2}-C_{i}^{2}\qquad\qquad e=E_{i}^{2}-D_{i}^{2}
\end{flalign*}

\end{prop}

\begin{prop}
\label{prop:C_2^5}Let $X_{p}$ be a K3 surface in the family $\mathcal{X}$,
for each $i\in[1,\dots,10]$ define $K_{i}$ as in Proposition \ref{thm:Field of definition of 32 conics}.
Then $\mathrm{Gal}(K_{i}/K)\cong C_{2}^{n}$ for some $0\leq n\leq5$
(that is $n$ copies of $\zz/2\zz$). In particular, $K_{i}=K(r_{1,i},r_{2,i},r_{\mu,i})$
(i.e., adjoining $\overline{r_{\mu,i}}$ is redundant).\end{prop}
\selectlanguage{english}%
\begin{proof}
We are going to show that if the polynomials (\ref{eq:first_root})
to (\ref{eq:last_root}) are irreducible then $\mbox{Gal}(K_{i}/K)\cong C_{2}^{5}$.
If any of the polynomials are not irreducible, then $\mbox{Gal}(K_{i}/K)$
is a subgroup of $C_{2}^{5}$, and hence must be $C_{2}^{n}$ for
some $0\leq n\leq5$. 

To do so we will use the resolvent method. Consider the group 
\[
\left\langle (12)(34)(56)(78),(13)(24)(57)(68),(15)(37)(26)(48)\right\rangle \leq S_{8}.
\]
Note that this is the group of translations of $(\zz/2\zz)^{3}$ (label
the eight vertices of a fundamental cube $1$ to $8$), hence it is
$C_{2}^{3}$. Let $x_{1},\dots,x_{8}$ be indeterminate variables,
then $S_{8}$ acts on them by $x_{i}\mapsto x_{\sigma(i)}$. Note
that the monomial $x_{1}x_{3}+x_{2}x_{4}+x_{5}x_{7}+x_{6}x_{8}$ is
$C_{2}^{3}$-invariant, so we can construct the resolvent polynomial
$R_{C_{2}^{3}}=\prod_{j=1}^{g}(X-P_{j})$ where $P_{j}$ are the elements
in the $S_{8}$-orbit of $x_{1}x_{3}+x_{2}x_{4}+x_{5}x_{7}+x_{6}x_{8}$.

We first consider the Galois group of $K(r_{1,i})$ over $K$, call
it $G$. As polynomial (\ref{eq:first_root}) has as roots the eight
different $z$ coordinates of the $16$ singular points, we have that
the minimal polynomial of $r_{1,i}$ factors as 
\[
(x-r_{1,i})(x+r_{1,i})\left(x-\frac{1}{r_{1,i}}\right)\left(x+\frac{1}{r_{1,i}}\right)(x-\overline{r}_{1,i})(x+\overline{r}_{1,i})\left(x-\frac{1}{\overline{r}_{1,i}}\right)\left(x+\frac{1}{\overline{r}_{1,i}}\right).
\]
If we substitute the $x_{j}$ with the $j$th root of the minimal
polynomial of $r_{1,i}$ (as ordered above), we find that
\begin{eqnarray*}
x_{1}x_{3}+x_{2}x_{4}+x_{5}x_{7}+x_{6}x_{8} & = & 4.
\end{eqnarray*}
 Hence in this case $R_{C_{2}^{3}}$ has a $K$-rational non-repeated
root, so $G\subseteq C_{2}^{3}$. But since the minimal polynomial
of $r_{1,i}$ is already of degree $8$, we must have $G\cong C_{2}^{3}$.
In fact $G$ is generated by $r_{1,i}\mapsto-r_{1,i}$, $r_{1,i}\mapsto\frac{1}{r_{1,i}}$
and $r_{1,i}\mapsto\overline{r}_{1,i}$, denote them by $\sigma_{2},\sigma_{3}$
and $\sigma_{4}$ respectively.

Next, we consider the Galois group of $K(r_{1,i},r_{2,i})$ over $K$.
We have that the minimal polynomial of $r_{2,i}$ is of degree $8$
(either through direct calculation, or the fact that $r_{2,i}$ solves
a quadratic in $r_{1,i}^{2}$ which itself solves a quartic). We can
find all the conjugates of $r_{2,i}$, by noting that if we let $\sigma_{2},\sigma_{3},\sigma_{4}$
act on equation (\ref{eq:second_root}), we get with $\pm r_{2,i}$
a total of $8$ conjugates. Furthermore, we know that $\pm r_{2,i}$
corresponds to the $y$-coordinate of the singular points which have
$z$-coordinate $r_{1,i}$. Similarly, $\sigma_{3}(\pm r_{2,i})$
corresponds to the $y$-coordinate of the singular points which have
$z$-coordinate $\frac{1}{r_{1,i}}$. Hence we know that the minimal
polynomial of $r_{2,i}$ factors as
\[
(x-r_{2,i})(x+r_{2,i})\left(x-\frac{1}{r_{2,i}}\right)\left(x+\frac{1}{r_{2,i}}\right)(x-\overline{r}_{2,i})(x+\overline{r}_{2,i})\left(x-\frac{1}{\overline{r}_{2,i}}\right)\left(x+\frac{1}{\overline{r}_{2,i}}\right),
\]
where $\overline{r}_{2,i}=\frac{r_{1,i}}{\overline{r}_{1,i}}r_{2,i}$.
As above, we can see that the Galois group of $K(r_{2,i})$ over $K$
is $C_{2}^{3}$, and in particular, we now know that the $K(r_{1,i},r_{2,i})/K$
is Galois. After having made some choice of sign on $\sigma_{i}(r_{2,i})$
for $2\leq i\leq4$, it is not hard to see that in fact $\mbox{Gal}(K(r_{1,i}r_{2,i})/K)\cong C_{2}^{4}$
generated by $\sigma_{1},\sigma_{2},\sigma_{3}$ and $\sigma_{4}$,
where $\sigma_{1}(r_{1,i})=r_{1,i}$ and $\sigma_{1}(r_{2,i})=-r_{2,i}$. 

Finally we look at $K(r_{\mu,i},\overline{r}_{\mu,i})$, and first
note that $r_{\mu,i}$ and $\overline{r}_{\mu,i}$ have the same minimal
polynomial over $K$. In fact, we have that the minimal polynomial
of $r_{\mu,i}$ factors as{\small{
\[
\left(x-r_{\mu,i}\right)\left(x+r_{\mu,i}\right)\left(x-\frac{r_{\mu,i}}{r_{1,i}^{2}}\right)\left(x+\frac{r_{\mu,i}}{r_{1,i}^{2}}\right)\left(x-\overline{r}_{\mu,i}\right)\left(x+\overline{r}_{\mu,i}\right)\left(x-\frac{\overline{r}_{\mu,i}}{\overline{r}_{1,i}^{2}}\right)\left(x+\frac{\overline{r}_{\mu,i}}{\overline{r}_{1,i}^{2}}\right).
\]
}}In this case, if we substitute the $x_{j}$ with the $j$th root
of the minimal polynomial of $r_{\mu,i}$ (as ordered above), we find
that 
\begin{eqnarray*}
x_{1}x_{3}+x_{2}x_{4}+x_{5}x_{7}+x_{6}x_{8} & = & 2\left(\frac{r_{\mu,i}^{2}}{r_{1,i}^{2}}+\frac{\overline{r}_{\mu,i}^{2}}{\overline{r}_{1,i}^{2}}\right)\\
 & = & -\frac{2}{\Delta}\left(2C_{1}+A_{1}\left(r_{1,i}^{2}+\frac{1}{r_{1,i}^{2}}+\overline{r}_{1,i}^{2}+\frac{1}{\overline{r}_{1,i}^{2}}\right)\right).
\end{eqnarray*}
Since $r_{1,i}^{2}$ solves a quartic polynomial whose other roots
are $\frac{1}{r_{1,i}^{2}},\overline{r}_{1,i}^{2},\frac{1}{\overline{r}_{1,i}^{2}}$,
we have that the above expression is in $K$. So $R_{H'}$ has a $K$-rational
non-repeated root, hence $H\subseteq H'\cong C_{2}^{3}$. But since
the minimal polynomial of $r_{\mu,i}$ is already of degree $8$,
we must have $H\cong C_{2}^{3}$ and $K(r_{\mu,i},\overline{r}_{\mu,i})\cong K(r_{\mu,i})$.

Hence since $[K_{i}:K]=2\cdot2\cdot8=32$, so we are looking for a
group of order $32$, which has $C_{2}^{4}$ as a subgroup. Note that
the element fixing $r_{1,i},r_{2,i}$ and sending $r_{\mu,i}\mapsto\frac{1}{r_{\mu,i}}$,
call it $\sigma_{5}$, has order $2$, but is not in the subgroup
$\left\langle \sigma_{1},\sigma_{2},\sigma_{3},\sigma_{4}\right\rangle \cong C_{2}^{4}$
(again, after having made some choice of sign on $\sigma_{i}(r_{\mu,i})$).
Furthermore, one can check that $\sigma_{5}$ commutes with $\sigma_{i}$
for $1\leq i\leq4$ (after having extended $\sigma_{i}$ properly
on $K_{i}$). Hence we have that $\mbox{Gal}(K_{i}/K)\cong C_{2}^{5}$. 
\end{proof}
\selectlanguage{british}%
The following lemma will allow us to find another way of expressing
$K_{i}$, which will help us finding $L$. While this lemma is quite
standard, the proof has been included as it details how one can construct
the field isomorphic of $K_{i}$.
\selectlanguage{english}%
\begin{lem}
\label{lem:Galois group of C_2^n}If $\mathrm{Gal}(L/K)\cong C_{2}^{n}$
for some $n$, then there exists $\Delta_{1},\dots,\Delta_{n}\in K$
whose image are linearly independent in the $\ff_{2}$-vector space
$K^{*}/(K^{*})^{2}$, such that $L\cong K\left(\sqrt{\Delta_{1}},\dots,\sqrt{\Delta_{n}}\right)$.\end{lem}
\begin{proof}
Let $\mbox{Gal}(L/K)=\left\langle \sigma_{1},\dots,\sigma_{n}|\sigma_{i}^{2}=(\sigma_{i}\sigma_{j})^{2}=1\right\rangle \cong C_{2}^{n}$
and let 
\[
\widetilde{\sigma_{i}}=\left\langle \sigma_{1},\dots,\sigma_{i-1},\sigma_{i+1},\dots,\sigma_{n}\right\rangle \cong C_{2}^{n-1}
\]
 (for each $i\in\{1,\dots,n\}$). Since, for each $i$, we have that
$[\mbox{Gal}(L/K):\widetilde{\sigma_{i}}]=2$, the fixed field of
$\widetilde{\sigma}_{i}$, $L^{\widetilde{\sigma_{i}}}$, is a degree
$2$ extension of $K$. Hence $L^{\widetilde{\sigma}_{i}}=K(\sqrt{\Delta_{i}})$
for some $\Delta_{i}\in K$. 

We prove that $[K(\sqrt{\Delta_{i}})(\sqrt{\Delta_{1}},\dots,\sqrt{\Delta_{i-1}}):K(\sqrt{\Delta_{1}},\dots,\sqrt{\Delta_{i-1}})]=2$
by showing that $\sqrt{\Delta_{i}}\notin K(\sqrt{\Delta_{1}},\dots,\sqrt{\Delta_{i-1}})$.
Suppose, that this was the case, then by considering minimal polynomial,
we can show that $\sqrt{\Delta_{i}}=\alpha\sqrt{\Delta_{i_{1}}\dots\Delta_{i_{s}}}$
for some $\alpha\in K$ and subset $\{i_{1},\dots,i_{s}\}$ of $\{1,\dots,i-1\}$,
i.e., $\Delta_{i}$ is not linearly independent of $\Delta_{1},\dots,\Delta_{i-1}$
in $K^{*}/(K^{*})^{2}$. Hence $K(\sqrt{\Delta_{i}})\cong K(\sqrt{\Delta_{i_{1}}\dots\Delta_{i_{s}}})$
and $\sigma_{i_{1}}\in\widetilde{\sigma_{i}}$ fixes $\sqrt{\Delta_{i_{1}}\dots\Delta_{i_{s}}}$.
But since $\sigma_{i_{1}}\in\widetilde{\sigma_{j}}$ for $j\in\{i_{2},\dots,i_{s}\}$,
we also have that $\sigma_{i_{1}}$ fixes $\sqrt{\Delta_{j}}$. So
$\sqrt{\Delta_{i_{1}}\dots\Delta_{i_{s}}}=\sigma_{i_{1}}(\sqrt{\Delta_{i_{1}}\dots\Delta_{i_{s}}})=\sigma_{i_{1}}(\sqrt{\Delta_{i_{1}}})\sqrt{\Delta_{i_{2}}\dots\Delta_{i_{s}}}$,
hence $\sigma_{i_{1}}$ fixes $\sqrt{\Delta_{i_{1}}}$. This is a
contradiction, since then $K(\sqrt{\Delta_{i_{1}}})$ is the fixed
field of $\widetilde{\sigma}_{i_{1}}\times\left\langle \sigma_{i_{1}}\right\rangle =\mbox{Gal}(L/K)$. 

As $L^{\widetilde{\sigma}_{i}}\subset L$, we have that $\sqrt{\Delta_{i}}\in L$.
So $K(\sqrt{\Delta_{1}},\dots,\sqrt{\Delta_{n}})\subset L$, but by
the previous paragraph and the tower law, we also have $[K(\sqrt{\Delta_{1}},\dots,\sqrt{\Delta_{n}}):K]=2^{n}$.
Hence $L\cong K(\sqrt{\Delta_{1}},\dots,\sqrt{\Delta_{n}})$. 
\end{proof}
This means, that for each of the fields $K_{i}$, we can find an isomorphic
field of the form $K(\sqrt{\Delta_{1,i}},\dots,\sqrt{\Delta_{5,i}})$,
and the compositum, i.e. $L$, will be $K(\sqrt{\Delta_{1,1}},\dots,\sqrt{\Delta_{5,10}})$. 
\begin{prop}
Let $p=[A,B,C,D,E]\in\pp^{4}$ not on the Segre cubic or singular
hyperplanes, then the $32$ conics lying on $X_{p}$ defined by the
point $q_{1}$ are defined over the field
\[
K_{1}\cong K\left(\sqrt{\Delta q_{+C}p_{-0}p_{+1}},\sqrt{\Delta q_{+C}p_{+0}p_{-1}},\sqrt{\Delta q_{+D}p_{+0}p_{-2}},\sqrt{\Delta q_{+D}p_{-0}p_{+2}},\sqrt{-\Delta q_{-E}p_{+1}p_{-2}}\right).
\]
\end{prop}
\begin{proof}
We can use Lemma \ref{lem:Galois group of C_2^n} to construct $K_{1}$.
We have $\mbox{Gal}(K_{1}/K)=\left\langle \sigma_{1},\sigma_{2},\sigma_{3},\sigma_{4},\sigma_{5}\right\rangle $
where $\sigma_{j}$ acts on $r_{1,1},r_{2,1},r_{\mu,1}$ according
to the following table
\[
\begin{array}{c|ccccc}
 & \sigma_{1} & \sigma_{2} & \sigma_{3} & \sigma_{4} & \sigma_{5}\\
\hline r_{1,1} & -r_{1,1} & \frac{1}{r_{1,1}} & \overline{r}_{1,1} & r_{1,1} & r_{1,1}\\
r_{2,1} & r_{2,1} & \overline{r}_{2,1} & \frac{1}{r_{2,1}} & -r_{2,1} & r_{2,1}\\
r_{\mu,1} & r_{\mu,1} & \frac{r_{\mu,1}}{r_{1,1}^{2}} & \overline{r}_{\mu,1} & r_{\mu,1} & -r_{\mu,1}
\end{array}
\]
(where $\overline{r}_{2,1}=\frac{r_{1,1}}{\overline{r}_{1,1}}r_{2,1}$).
We calculate the fixed field of $\widetilde{\sigma_{1}}=\left\langle \sigma_{2},\dots,\sigma_{5}\right\rangle $
by considering the expression $r_{1,1}+\frac{1}{r_{1,1}}+\overline{r}_{1,1}+\frac{1}{\overline{r}_{1,1}}$
which is fixed under $\sigma_{j}$, $j\in\{2,\dots,5\}$ but not under
$\sigma_{1}$. Hence, upon calculating the discriminant of the minimal
polynomial (after checking it is quadratic) of such an expression,
we have that the fixed field of $\widetilde{\sigma}_{1}$ is $K\left(\sqrt{p_{-0}p_{+1}p_{+2}p_{-2}q_{+D}\left(-q_{-E}\right)}\right)$
. Similarly we can use the following expressions to calculate the
respective fixed fields:
\begin{itemize}
\item $r_{1,1}^{2}+\overline{r}_{1,1}^{2}$ for $\widetilde{\sigma}_{2}$
giving $K\left(\sqrt{p_{+1}p_{-1}p_{+2}p_{-2}}\right)$,
\item $r_{1,1}^{2}+\frac{1}{r_{1,1}^{2}}$ for $\widetilde{\sigma}_{3}$
giving $K(\sqrt{p_{+0}p_{-0}p_{+1}p_{-1}})$,
\item $r_{2,1}+\frac{1}{r_{2,1}}+\overline{r}_{2,1}+\frac{1}{\overline{r}_{2,1}}$
for $\widetilde{\sigma}_{4}$ giving $K(\sqrt{p_{+0}p_{+1}p_{-1}p_{-2}q_{+C}\left(-q_{-E}\right)})$,
\item $r_{\mu,1}+\frac{r_{\mu,1}}{r_{1,1}^{2}}+\overline{r}_{\mu,1}+\frac{\overline{r}_{\mu,1}}{\overline{r}_{1,1}^{2}}$
for $\widetilde{\sigma}_{5}$ giving $K(\sqrt{\Delta p_{+0}p_{-0}q_{+C}q_{+D}(-q_{-E})})$.
\end{itemize}
Putting all of this together and rearranging, we get the required
result.\end{proof}
\begin{thm}
\label{thm:Field Of Def 320}Let $p=[A,B,C,D,E]\in\pp^{4}$ be a general
point not lying on the Segre cubic or the $15$ singular hyperplanes
and let $L$ be the field where the $320$ conics of $X_{p}$ are
defined. Then $\mathrm{Gal}(L/K)\cong C_{2}^{10}$.\end{thm}
\begin{proof}
The first step is to calculate $K_{i}$ for $i\in\{2,\dots,10\}$
in terms of squares roots of elements in $K$. This is done by doing
same calculations as the above proposition with different $q_{i}$
(and hence $r_{1,i}$, $r_{2,i}$, $r_{\mu,i}$). \foreignlanguage{british}{We
find, up to rearrangements
\[
K_{2}=K\left(\sqrt{\Delta q_{+C}p_{-0}p_{+1}},\sqrt{\Delta q_{+C}p_{+0}p_{-1}},\sqrt{\Delta q_{+E}p_{+0}p_{-3}},\sqrt{\Delta q_{+E}p_{-0}p_{+3}},\sqrt{-\Delta q_{-D}p_{+1}p_{-3}}\right),
\]
\[
K_{3}=K\left(\sqrt{\Delta q_{+D}p_{-0}p_{+2}},\sqrt{\Delta q_{+D}p_{+0}p_{-2}},\sqrt{\Delta q_{+E}p_{+0}p_{-3}},\sqrt{\Delta q_{+E}p_{-0}p_{+3}},\sqrt{-\Delta q_{-C}p_{+2}p_{-3}}\right),
\]
\[
K_{4}=K\left(\sqrt{-\Delta q_{-D}p_{+1}p_{-3}},\sqrt{-\Delta q_{-D}p_{-1}p_{+3}},\sqrt{-\Delta q_{-E}p_{-1}p_{+2}},\sqrt{-\Delta q_{-E}p_{+1}p_{-2}},\sqrt{-\Delta q_{-C}p_{+2}p_{-3}}\right),
\]
\[
K_{5}=K\left(\sqrt{-\Delta Aq_{+E}q_{-E}},\sqrt{-\Delta Aq_{+D}q_{-D}},\sqrt{\Delta q_{+D}p_{+0}p_{-2}},\sqrt{\Delta q_{+E}p_{+0}p_{-3}},\sqrt{-\Delta q_{-E}p_{+1}p_{-2}}\right),
\]
\[
K_{6}=K\left(\sqrt{-\Delta Aq_{+E}q_{-E}},\sqrt{-\Delta Aq_{+D}q_{-D}},\sqrt{\Delta q_{+D}p_{-0}p_{+2}},\sqrt{\Delta q_{+E}p_{-0}p_{+3}},\sqrt{-\Delta q_{-E}p_{-1}p_{+2}}\right),
\]
\[
K_{7}=K\left(\sqrt{-\Delta Aq_{+C}q_{-C}},\sqrt{-\Delta Aq_{+E}q_{-E}},\sqrt{\Delta q_{+C}p_{+0}p_{-1}},\sqrt{\Delta q_{+E}p_{+0}p_{-3}},\sqrt{-\Delta q_{-E}p_{+1}p_{-2}}\right),
\]
\[
K_{8}=K\left(\sqrt{-\Delta Aq_{+C}q_{-C}},\sqrt{-\Delta Aq_{+E}q_{-E}},\sqrt{\Delta q_{+C}p_{-0}p_{+1}},\sqrt{\Delta q_{+E}p_{-0}p_{+3}},\sqrt{-\Delta q_{-E}p_{-1}p_{+2}}\right),
\]
\[
K_{9}=K\left(\sqrt{-\Delta Aq_{+C}q_{-C}},\sqrt{-\Delta q_{+D}q_{-D}},\sqrt{\Delta q_{+C}p_{+0}p_{-1}},\sqrt{\Delta q_{+D}p_{+0}p_{-2}},\sqrt{-\Delta q_{-D}p_{+1}p_{-3}}\right),
\]
\[
K_{10}=K\left(\sqrt{-\Delta Aq_{+C}q_{-C}},\sqrt{-\Delta q_{+D}q_{-D}},\sqrt{\Delta q_{+C}p_{-0}p_{+1}},\sqrt{\Delta q_{+D}p_{-0}p_{+2}},\sqrt{-\Delta q_{-D}p_{-1}p_{+3}}\right).
\]
}

\selectlanguage{british}%
Then as the $320$ conics of $X_{p}$ are defined over the compositum
of $K_{1},\dots,K_{10}$, we see that $L$ is the field extension
\foreignlanguage{english}{
\[
K\bigg(\sqrt{-\Delta Aq_{+C}q_{-C}},\sqrt{-\Delta Aq_{+D}q_{-D}},\sqrt{-\Delta Aq_{+E}q_{-E}},\sqrt{\Delta q_{+C}p_{+0}p_{-1}},\sqrt{\Delta q_{+C}p_{-0}p_{+1}},
\]
\[
\sqrt{\Delta q_{+D}p_{+0}p_{-2}},\sqrt{\Delta q_{+D}p_{-0}p_{+2}},\sqrt{\Delta q_{+E}p_{+0}p_{-3}},\sqrt{\Delta q_{+E}p_{-0}p_{+3}},\sqrt{-\Delta q_{-C}p_{+2}p_{-3}}\bigg).
\]
Now for a general point $p$, the $10$ square root are distinct and
not in $K$, i.e., $[L:K]=2^{10}$ (this holds, for example, in the
case $K=\qq$ and $[A,B,C,D,E]=[1,87,15,39,21]$). Hence $\mathrm{Gal}(L/K)\cong C_{2}^{10}$.}
\end{proof}
\selectlanguage{british}%

\section{Monodromy Group\label{sec:Monodromy-Group}}

Note that the field of definition of the $320$ conics involved equations
defining when a surface $X$ in $\mathcal{X}$ is singular (cf Theorem
\ref{thm:Field Of Def 320} and Proposition \ref{prop:Singular Conditions}).
We explain this by studying the Monodromy group of the conics over
a general non-singular K3 surface in $\mathcal{X}$. First we briefly
recall what the Monodromy group is.

Let $Z$ be an algebraic variety with $\pi:Z\to X$ a surjective finite
map of degree $d>0$. Let $p\in X$ be a general point of $X$ and
$\pi^{-1}(p)=\{q_{0},\dots,q_{d-1}\}$ be its fibre. Let $U\subset X$
be a suitable small Zariski open subset of $X$ and $V=\pi^{-1}(U)$.
Then for any loop $\lambda:[0,1]\to U$ based at $p$, and any point
$q_{i}\in\pi^{-1}(p)$, there exists a unique path $\widetilde{\lambda}_{i}$
in $V$ such that $\pi(\widetilde{\lambda}_{i})=\lambda$ and $\widetilde{\lambda}_{i}(0)=q_{i}$.
So we may define a permutation $\sigma_{\lambda}$ of $\pi^{-1}(p)$
by sending each point $q_{i}$ to $\widetilde{\lambda}_{i}(1)=q_{j}$
(for some $j$). Since $\phi_{\lambda}$ only depends on the homotopy
class of $\lambda$, we have a homomorphism $\pi_{1}(U,p)\to S_{d}$.
The image of this homomorphism is called the \emph{Monodromy group}
of the map $\pi$.

In our case we want the variety $Z$ to parametrise the quartic surfaces
of $\mathcal{X}$ with the $320$ conics on them, and study the permutations
of the conics as we draw loops on $\mathcal{X}$. We first study a
simpler problem, namely we will look at the Monodromy group of the
$16$ planes defined by the point $q_{1}=[1,0,-2,-2,2]\in\pp^{4}$.
We shall denote this set of $16$ planes by $T_{1}$, (they are the
tropes described in the proof of Theorem \ref{thm:320-Conics})
\selectlanguage{english}%
\begin{lem}
\label{lem:The_planes_q1}Let $p=[A,B,C,D,E]\in\pp^{4}$ be a general
point not lying on the Segre cubic or the $15$ singular hyperplanes.
The set $T_{1}$ is $\{\gamma(r_{0,1}x+r_{1,1}y+r_{2,1}z+r_{3,1}w)|\gamma\in\Gamma\}$
where
\selectlanguage{british}%
\begin{itemize}
\item {\small{$r_{0,1}=2^{3}B\sqrt{-q_{+D}q_{+C}q_{-E}}$ ,}}{\small \par}
\item {\small{$r_{1,1}=\sqrt{q_{+C}}\left(\sqrt{p_{-2}p_{-0}p_{+2}p_{+1}}+p_{+2}\sqrt{p_{+1}p_{+0}}+\sqrt{p_{-2}p_{-1}p_{+2}p_{+0}}+p_{-2}\sqrt{p_{-1}p_{-0}}\right)$,}}{\small \par}
\item {\small{$r_{2,1}=\sqrt{q_{+D}}\left(\sqrt{p_{-1}p_{-0}p_{+1}p_{+2}}+p_{+1}\sqrt{p_{+2}p_{+0}}+\sqrt{p_{-1}p_{-2}p_{+1}p_{+0}}+p_{-1}\sqrt{p_{-2}p_{-0}}\right)$,}}{\small \par}
\item {\small{$r_{3,1}=-\sqrt{-q_{-E}}\left(\sqrt{p_{-0}p_{-1}p_{+0}p_{+2}}+p_{+0}\sqrt{p_{+2}p_{+1}}+\sqrt{p_{-0}p_{-2}p_{+0}p_{+1}}+p_{-0}\sqrt{p_{-2}p_{-1}}\right)$,}}{\small \par}
\end{itemize}
where $p_{\pm i},q_{\pm\alpha}$ are the equations in Proposition
\ref{prop:Singular Conditions}.\end{lem}
\selectlanguage{british}%
\begin{proof}
Theorem \ref{thm:Field of definition of 32 conics} already gives
an expression for the planes, but \foreignlanguage{english}{we can
use the isomorphism $K(r_{1,1},r_{2,1},r_{\mu,1})\cong K\left(\sqrt{\Delta q_{+C}p_{-0}p_{+1}},\sqrt{\Delta q_{+C}p_{+0}p_{-1}},\sqrt{\Delta q_{+D}p_{+0}p_{-2}},\sqrt{\Delta q_{+D}p_{-0}p_{+2}},\sqrt{-\Delta q_{-E}p_{+1}p_{-2}}\right)$
to rewrite the singular point $[r_{3,1},r_{2,1},r_{1,1},1]$ in terms
of linear combinations of square roots. }

\selectlanguage{english}%
That is, let $r$ be any of those roots, we know that $r$ solves
a degree $8$ polynomial, whose terms are all even and $K(r)\cong K(\sqrt{\Delta_{1}},\sqrt{\Delta_{2}},\sqrt{\Delta_{3}})$
for some $\Delta_{i}$'s. So let $r=a_{0}+a_{1}\sqrt{\Delta_{1}}+a_{2}\sqrt{\Delta_{2}}+\dots+a_{7}\sqrt{\Delta_{1}\Delta_{2}\Delta_{3}}$.
The Galois group of $K(\sqrt{\Delta_{1}},\sqrt{\Delta_{2}},\sqrt{\Delta_{3}})$
is naturally generated by $\sigma_{1},\sigma_{2},\sigma_{3}$ where
$\sigma_{i}(\sqrt{\Delta_{j}})=\begin{cases}
-\sqrt{\Delta_{j}} & i=j\\
\sqrt{\Delta_{j}} & j\neq i
\end{cases}$. On one hand we know that the minimal polynomial of $r$ \foreignlanguage{british}{factorises}
as 
\[
(x-r)(x+r)\left(x-\frac{1}{r}\right)\left(x+\frac{1}{r}\right)(x-\overline{r})(x+\overline{r})\left(x-\frac{1}{\overline{r}}\right)\left(x+\frac{1}{\overline{r}}\right),
\]
and on the other hand, it \foreignlanguage{british}{factorises} 
\[
\prod_{g\in\mathrm{Gal}(K(r)/K)}(x-g(a_{0}+a_{1}\sqrt{\Delta_{1}}+\dots+a_{7}\sqrt{\Delta_{1}\Delta_{2}\Delta_{3}})).
\]
Now we have $0=r-r+\dots+\frac{1}{\overline{r}}-\frac{1}{\overline{r}}=8a_{0}$,
hence $a_{0}=0$. Without loss of generality, suppose that $\sigma_{1}(r)=-r$,
then $0=r-r=a_{2}\sqrt{\Delta_{2}}+a_{3}\sqrt{\Delta_{3}}+a_{6}\sqrt{\Delta_{2}\Delta_{3}}$,
so $r=\sqrt{\Delta_{1}}(a_{1}+a_{4}\sqrt{\Delta_{2}}+a_{5}\sqrt{\Delta_{3}}+a_{7}\sqrt{\Delta_{2}\Delta_{3}})$.
If we suppose $\sigma_{2}(r)=\frac{1}{r}$ and $\sigma_{3}(r)=\overline{r}$,
then note that
\begin{eqnarray*}
r+\frac{1}{r}+\overline{r}+\frac{1}{\overline{r}} & = & 4a_{1}\sqrt{\Delta_{1}},\\
r-\frac{1}{r}+\overline{r}-\frac{1}{\overline{r}} & = & 4a_{4}\sqrt{\Delta_{1}\Delta_{2}},\\
r+\frac{1}{r}-\overline{r}-\frac{1}{\overline{r}} & = & 4a_{5}\sqrt{\Delta_{1}\Delta_{3}},\\
r-\frac{1}{r}-\overline{r}+\frac{1}{\overline{r}} & = & 4a_{7}\sqrt{\Delta_{1}\Delta_{2}\Delta_{3}}.
\end{eqnarray*}
Hence we can easily work out $a_{1},a_{4},a_{5},a_{7}$, and therefore
$r$, by calculating the minimal polynomial of the above 4 expressions.
Applying that theory to $r_{3,1}$, $r_{2,1}$ and $r_{1,1}$ in turn,
after some rearrangment, we get the required result.
\end{proof}
\selectlanguage{english}%
Hence, to study the Monodromy group of those $16$ planes on a K3
surface, we need the object $\mathcal{Z}$ defined by \foreignlanguage{british}{
\[
\left\{ \left([A,B,C,D,E],[a,b,c,d]\right)\,\big|\,[a,b,c,d]\in\{\gamma\left([r_{0,1},r_{1,1},r_{2,1},r_{3,1}]\right):\gamma\in\Gamma\}\right\} \subset\mathbb{P}^{4}\times\mathbb{P}^{3}.
\]
Note that we set up $\pp_{[a,b,c,d]}^{3}$ to be the dual of $\pp_{[x,y,z,w]}^{3}$,
that is a point $[a,b,c,d]\in\pp_{[a,b,c,d]}^{3}$ represent the plane
$ax+by+cz+dw=0$ in $\pp_{[x,y,z,w]}^{3}$ which intersects $X_{[A,B,C,D,E]}$
in two conics. Now by the above lemma, $r_{i,1}$ involves square
roots and hence $\mathcal{Z}$ is not a variety. So instead of looking
at the planes defined by the point $[A,B,C,D,E]$, we will look at
the points $[A,B,C,D,E]$ that can be defined by a given plane. Pick
a point $[a,b,c,d]\in\pp^{3}$ and let }
\selectlanguage{british}%
\begin{itemize}
\item $A_{1}=(bc+ad)(bc-ad)(ac+bd)(ac-bd)(cd+ab)(cd-ab)$,
\item $B_{1}=2abcd(-a^{2}-b^{2}+c^{2}+d^{2})(-a^{2}+b^{2}+c^{2}-d^{2})(a^{2}-b^{2}+c^{2}-d^{2})(a^{2}+b^{2}+c^{2}+d^{2})$,
\item $C_{1}=(bc+ad)(bc-ad)(ac+db)(ac-bd)(a^{4}+b^{4}-c^{4}-d^{4})$,
\item $D_{1}=(bc+ad)(bc-ad)(cd+ab)(cd-ab)(-a^{4}+b^{4}-c^{4}+d^{4})$,
\item $E_{1}=(ac+db)(ac-bd)(cd+ab)(cd-ab)(a^{4}-b^{4}-c^{4}+d^{4}).$
\end{itemize}
Then $[A,B,C,D,E]$ defines a surface which has the plane $ax+by+cz+dy=0$
given by the point $q_{1}$ if and only if the point $[A,B,C,D,E]$
lies on the line joining $[A_{1},B_{1},C_{1},D_{1},E_{1}]$ and $q_{1}$.
So let us consider the variety $\mathcal{V}$ defined by
\[
\{E_{1}(2A+D)+D_{1}(2A-E)-2A_{1}(E+D)=C_{1}(2A+D)-D_{1}(2A+C)+2A_{1}(D-C)
\]
\[
=B_{1}(2A+D)-D_{1}B-2A_{1}B=0\}\subseteq\pp_{[A,B,C,D,E]}^{4}\times\pp_{[a,b,c,d]}^{3}.
\]

The variety $\mathcal{V}$ has the following properties:
\begin{itemize}
\item over any point of the dual of $\mathcal{L}$, we have a copy of $\mathbb{P}_{[A,B,C,D,E]}^{4}$,
\item over the conic $a^{2}-b^{2}-c^{2}+d^{2}$, we have a copy of $\mathbb{P}_{[A,B,C,D,E]}^{4}$
,
\item over the point $[1:0:-2:-2:2]$, we have a copy of $\mathbb{P}_{[a:b:c:d]}^{3}$,
\item everywhere else, the variety $\mathcal{V}$ coincides with to $\mathcal{Z}$.
\end{itemize}
We note that the dual of $\mathcal{L}\subset\pp_{[x,y,z,w]}^{3}$
is itself $\mathcal{L}\subset\pp_{[a,b,c,d]}^{3}$ (by making the
correspondence $x\leftrightarrow a,\dots,w\leftrightarrow d$), this
is because each $L_{i}$ is dual to $\overline{L}_{i}$. Similarly,
we have that each quadric $Q_{i}\subset\pp_{[x,y,z,w]}^{3}$ can be
identified with itself $Q_{i}\subset\pp_{[a,b,c,d]}^{3}$. 
\begin{lem}
Let $\mathcal{V}\subset\mathbb{P}_{[A,B,C,D,E]}^{4}\times\mathbb{P}_{[a:b:c:d]}^{3}$
be as above and $p_{1},p_{2}$ the projective maps $\mathcal{V}\mapsto\pp_{[a,b,c,d]}^{3}$
and $\mathcal{V}\mapsto\pp_{[A,B,C,D,E]}^{4}$ respectively. Then
the projective map $p_{2}:\mathcal{V}\to\mathbb{P}_{[A:B:C:D:E]}^{4}$
is smooth away from the points $p$ such that $p_{1}(p)$ lies on
$10$ quadrics $Q_{i}$, or $p_{2}(p)$ is the point $q_{1}=[1:0:-2:-2:2]$.\end{lem}
\begin{proof}
Note that the union of the $15$ pairs of lines \emph{$\mathcal{L}$}
are contained in the $10$ quadratics $Q_{i}$. Away from these $15$
pairs of lines, once we have fixed $[a:b:c:d]$, we have that the
point $[A:B:C:D:E]$ is of the form 
\[
[\mu a_{p}+(1-\mu),\mu b_{p},\mu c_{p}-2(1-\mu),\mu d_{p}-2(1-\mu),\mu e_{p}+2(1-\mu)]
\]
 for some $\mu\in K$. So we want to show that the Jacobian matrix
\[
\left(\frac{\partial g_{i}}{\partial x_{j}}\bigg|_{A=\mu a_{p}+(1-\mu),B=\mu b_{p},C=\mu c_{p}-2(1-\mu),D=\mu d_{p}-2(1-\mu),E=\mu e_{p}+2(1-\mu)}\right)_{i,j}
\]
 has rank $3$. This is equivalent to showing that the determinant
of at least one of the $4$ matrices obtained from deleting a row
in the Jacobian is non-zero. We can calculate that the $4$ determinants
are $\mu^{3}aF$, $\mu^{3}bF$, $\mu^{3}cF$, $\mu^{3}dF$ where {\small{
\begin{eqnarray*}
F & = & (bc-ad)^{4}(bc+ad)^{4}(ac-bd)^{2}(ac+bd)^{2}(ab-cd)^{4}(ab+cd)^{4}\\
 &  & \cdot(a^{2}-b^{2}-c^{2}+d^{2})^{6}(a^{2}-b^{2}+c^{2}-d^{2})^{2}(a^{2}+b^{2}-c^{2}-d^{2})^{4}(a^{2}+b^{2}+c^{2}+d^{2})^{2}.
\end{eqnarray*}
}}Note that $F$ is a product of the $10$ quadratics, $Q_{i}$, and
hence cannot be $0$. While if $\mu=0$, then the surface $[A,B,C,D,E]$
is $[1:0:-2:-2:2]$. If $F\neq0$ and $\mu\neq0$, then one of the
$4$ determinants must be non-zero, hence the projection map is smooth
at that place.
\end{proof}
So we only need to worry about points lying on one of the $10$ quadrics
$Q_{i}$. Since a point on it either lies on $\mathcal{L}$ or the
point gives rise to one of the 10 nodes, from our construction (and
not the variety $\mathcal{V}$), such cases correspond to the K3 surface
lying on one of the $15$ singular hyperplanes. Hence to study the
Monodromy group of $16$ planes defined by $q_{1}$, we need to look
at the $15$ singular hyperplanes. We will study this on the level
of the object $\mathcal{Z}$ and not the variety $\mathcal{V}$.
\begin{prop}
\label{thm:Monodromy}There exists a loop in $\mathbb{P}^{4}$ that:
\begin{enumerate}
\item goes around the singular hyperplane $q_{+C}$ and avoids all of the
$15$ singular hyperplanes,
\item changes the sign of $\sqrt{-q_{+C}}$, which in turn sends the plane
$r_{0,1}x+r_{1,1}y+r_{2,1}z+r_{3,1}w=0$ to the plane $\gamma_{3}\left(r_{0,1}x+r_{1,1}y+r_{2,1}z+r_{3,1}w\right)=0$.
(Where $\gamma_{3}\in\Gamma<\Omega$ as define in Section \ref{sec:Background}).
\end{enumerate}
\end{prop}
\begin{proof}
Pick a point $[A:B:C:D:E]\in\mathbb{P}^{4}$ which does not represent
a singular K3 surface, and note that $C\in K\subset\mathbb{C}$ can
be written uniquely as $-2A+re^{i\phi}$ for some $r\in\mathbb{R}_{>0}$
and $\phi\in[0,2\pi)$. Define a loop $\widetilde{\lambda}(t)=[\widetilde{\lambda}_{A}(t),\widetilde{\lambda}_{B}(t),\widetilde{\lambda}_{C}(t),\widetilde{\lambda}_{D}(t),\widetilde{\lambda}_{E}(t)]$
($0\leq t\leq3$) as $\widetilde{\lambda}_{C}=-2A+f(t)$ and $\widetilde{\lambda}_{j}=j$
for all $t\in[0,3]$ and $j\in\{A,B,D,E\}$, where $f$ is composed
of the following $3$ segments:
\[
f(t)=\begin{cases}
(\rho t+r(1-t))e^{i\phi} & t\in[0,1]\\
\rho e^{i\phi+i(t-1)2\pi} & t\in[1,2]\\
(\rho(3-t)+r(t-2))e^{i\phi} & t\in[2,3]
\end{cases}
\]
and where $\rho\in\rr_{>0}$ satisfies {\small{
\[
\rho<\min\left\{ \left|B+2D+2E\right|,\left|-B+2D+2E\right|,\left|8A+B+2D-2E\right|,\left|8A-B+2D-2E\right|\right\} .
\]
}}Now consider the point $[r_{0,1},r_{1,1},r_{2,1},r_{3,1}]$ (where
we have fixed a root for each square roots) under this loop. As $t$
changes, the $10$ equations occurring in $r_{0,1},r_{1,1},r_{2,1},r_{3,1}$
are affected in the following ways:
\begin{itemize}
\item $q_{+C}=f(t)$,
\item $q_{+D}$, $q_{-E}$, $B$ all stay the same,
\item $p_{+0}=-p_{+1}=B+2D+2E+2f(t)$,
\item $p_{-0}=-p_{+1}=-B+2D+2E+2f(t)$,
\item $p_{+2}=8A+B+2D-2E-2f(t)$,
\item $p_{-2}=8A-B+2D-2E-2f(t)$.
\end{itemize}
For ease of argument we assume that for $0\leq t\leq1$, none of $\{p_{+0},p_{-0},p_{+1},p_{-1},p_{+2},p_{-2}\}$
are $0$ (if they are, the argument can be changed by slightly curving
the first segment instead of using a straight line). So for the first
segment, we can see that nothing remarkable happens. During the second
segment, we have chosen $\rho$ small enough so that none of $p_{+0},p_{-0},p_{+1},p_{-1},p_{+2}$
and \foreignlanguage{english}{$p_{-2}$} are $0$, but we see that
$\sqrt{q_{+C}}$ is affected. Indeed, if we choose the square root
of $e^{i\phi}$ to be $e^{\frac{i\phi}{2}}$, we see that $\sqrt{q_{+c}}=\sqrt{\rho}e^{\frac{i\phi}{2}+i(t-1)\pi}$.
Hence at $t=1$, $\sqrt{q_{+C}}$ is positive, but at $t=2$ the sign
has changed. Note that the third segment is the same as the first
segment but backwards.

Finally, one can see that by changing the sign of $\sqrt{q_{+C}}$,
we have $r_{0,1}\mapsto-r_{0,1}$, $r_{1,1}\mapsto-r_{1,1}$, $r_{2,1}\mapsto r_{2,1}$
and $r_{3,1}\mapsto r_{3,1}$. Hence the plane $r_{0,1}x+r_{1,1}y+r_{2,1}z+r_{3,1}w=0$
gets mapped to the plane $r_{0,1}x+r_{1,1}y-r_{2,1}z-r_{3,1}w=0=\gamma_{3}(r_{0,1}x+r_{1,1}y+r_{2,1}z+r_{3,1}w)$.
\end{proof}
A very similar argument works for the singular hyperplanes $q_{+D}$,
$q_{-E}$, $p_{+0}$, $p_{-0}$, $p_{+1}$, $p_{-1}$, $p_{+2}$ and
$p_{-2}$. For the singular hyperplanes $A,q_{-C},q_{-D},q_{+E},p_{+3}$
and $p_{-3}$, we note that either $[r_{0},r_{1},r_{2},r_{3}]$ are
completely unaffected, or see by direct calculations that we still
have $16$ different planes when plucking in $A=0$, or $q_{-C}=0$,$\dots$
etc. 
\begin{notation*}
Out of the $15$ singular hyperplanes, the point $q_{1}$ lies on
$9$ of them, namely $q_{+C}$, $q_{+D}$, $q_{-E}$, $p_{+0}$, $p_{-0}$,
$p_{+1}$, $p_{-1}$, $p_{+2}$ and $p_{-2}$. We shall denote this
set by $\Sigma_{q_{1}}$. \end{notation*}
\begin{prop}
The Monodromy group of set $T_{1}$ is isomorphic to $\Gamma$ and
hence $C_{2}^{4}$.\end{prop}
\begin{proof}
By the above discussion, the permutations of the $16$ planes come
from changing the sign of the square roots $\sqrt{\Delta_{i}}$ for
$\Delta_{i}\in\Sigma_{q_{1}}$. By direct calculation, we see that:
\begin{itemize}
\item $\sqrt{q_{+C}}\mapsto-\sqrt{q_{+C}}$ corresponds to $r_{0,1}x+r_{1,1}y+r_{2,1}z+r_{3,1}w\mapsto\gamma_{3}(r_{0,1}x+r_{1,1}y+r_{2,1}z+r_{3,1}w)$
\item $\sqrt{q_{+D}}\mapsto-\sqrt{q_{+D}}$ corresponds to $r_{0,1}x+r_{1,1}y+r_{2,1}z+r_{3,1}w\mapsto\gamma_{4}(r_{0,1}x+r_{1,1}y+r_{2,1}z+r_{3,1}w)$
\item $\sqrt{-q_{-E}}\mapsto-\sqrt{-q_{-E}}$ corresponds to $r_{0,1}x+r_{1,1}y+r_{2,1}z+r_{3,1}w\mapsto\gamma_{3}\gamma_{4}(r_{0,1}x+r_{1,1}y+r_{2,1}z+r_{3,1}w)$
\item $\sqrt{p_{+0}}\mapsto-\sqrt{p_{+0}}$ corresponds to $r_{0,1}x+r_{1,1}y+r_{2,1}z+r_{3,1}w\mapsto\gamma_{1}\gamma_{2}(r_{0,1}x+r_{1,1}y+r_{2,1}z+r_{3,1}w)$
\item $\sqrt{p_{-0}}\mapsto-\sqrt{p_{-0}}$ corresponds to $r_{0,1}x+r_{1,1}y+r_{2,1}z+r_{3,1}w\mapsto\gamma_{1}\gamma_{2}\gamma_{3}\gamma_{4}(r_{0,1}x+r_{1,1}y+r_{2,1}z+r_{3,1}w)$
\item $\sqrt{p_{+1}}\mapsto-\sqrt{p_{+1}}$ corresponds to $r_{0,1}x+r_{1,1}y+r_{2,1}z+r_{3,1}w\mapsto\gamma_{2}\gamma_{3}\gamma_{4}(r_{0,1}x+r_{1,1}y+r_{2,1}z+r_{3,1}w)$
\item $\sqrt{p_{-1}}\mapsto-\sqrt{p_{-1}}$ corresponds to $r_{0,1}x+r_{1,1}y+r_{2,1}z+r_{3,1}w\mapsto\gamma_{2}\gamma_{3}(r_{0,1}x+r_{1,1}y+r_{2,1}z+r_{3,1}w)$
\item $\sqrt{p_{+2}}\mapsto-\sqrt{p_{+2}}$ corresponds to $r_{0,1}x+r_{1,1}y+r_{2,1}z+r_{3,1}w\mapsto\gamma_{1}\gamma_{3}\gamma_{4}(r_{0,1}x+r_{1,1}y+r_{2,1}z+r_{3,1}w)$
\item $\sqrt{p_{-2}}\mapsto-\sqrt{p_{-2}}$ corresponds to $r_{0,1}x+r_{1,1}y+r_{2,1}z+r_{3,1}w\mapsto\gamma_{1}\gamma_{4}(r_{0,1}x+r_{1,1}y+r_{2,1}z+r_{3,1}w).$
\end{itemize}
Hence, we see that the Monodromy group $T_{1}$ is isomorphic to $\Gamma$.
\end{proof}
Next we want to calculate the Monodromy group of the $160$ planes
which intersect a surface in $\mathcal{X}$ to give the $320$ conics.
First we will calculate the Monodromy group of the $16$ planes associated
to $q_{i}$ for each $i\in[2,\dots,10]$, which we denote each set
by $T_{i}$. To do so, we will use the group $\Omega$ acting on $\pp^{3}\times\pp^{4}$.
This group permutes the $15$ singular planes and the $10$ points
$q_{1},\dots,q_{10}$ as follows (using permutation notation):
\begin{itemize}
\item $\phi_{1}$ acts as $(p_{+0},p_{-0})(p_{+1},p_{-1})(p_{+2},p_{-2})(p_{+3},p_{-3})$
and as $(q_{5},q_{6})(q_{7},q_{8})(q_{9},q_{10})$,
\item $\phi_{2}$ acts as $(q_{+D},q_{+E})(q_{-D},q_{-E})(p_{+2},p_{+3})(p_{-2},p_{-3})$
and as $(q_{1},q_{2})(q_{7},q_{9})(q_{8},q_{10})$,
\item $\phi_{3}$ acts as $(q_{+C},q_{+D})(q_{-C},q_{-D})(p_{+1},p_{+2})(p_{-1},p_{-2})$
and as $(q_{2},q_{3})(q_{5},q_{7})(q_{6},q_{8})$,
\item $\phi_{4}$ acts as $(q_{+D},q_{-D})(q_{+E},q_{-E})(p_{+0},p_{-1})(p_{-0},p_{+1})(p_{+2},p_{-3})(p_{-2},p_{+3})$
and as $(q_{1},q_{2})(q_{3},q_{4})(q_{5},q_{6})$,
\item $\phi_{5}$ acts as $(A,q_{+C})(q_{+D},p_{+0})(q_{-D},p_{-1})(q_{+E},p_{-0})(q_{-E},p_{+1})(p_{+2},p_{-3})$
and as $(q_{1},q_{5})(q_{2},q_{6})(q_{7},q_{10})$.
\end{itemize}
Hence we can find the Monodromy group of $T_{i}$ from the Monodromy
group of $T_{1}$. Pick an element $\phi\in\Omega$ which permutes
$q_{1}$ and $q_{i}$, then the element of the Monodromy group associated
to the singular hyperplane $H$ is $\phi^{-1}\cdot\gamma_{\phi(H)}\cdot\phi$,
where $\gamma_{\phi(H)}$ is the element of the Monodromy group of
$T_{1}$. Note that since $\Gamma$ is normal in $\Omega$, we necessarily
end up with an element of $\Gamma$.
\begin{example*}
We work out explicitly some of the cases for the point $q_{2}$. We
use the element $\phi_{2}$ which permutes the point $q_{1}$ and
$q_{2}$.

Since $\phi_{2}(A)=A$, the element corresponding to the hyperplane
$A$ in the Monodromy group of $T_{2}$ is $\phi_{2}^{-1}\gamma_{A}\phi_{2}=\phi_{2}^{-1}\cdot\id\cdot\phi_{2}=\id$.

Since $\phi_{2}(q_{+E})=q_{+D}$, the element corresponding to the
hyperplane $q_{+E}$ in the Monodromy group of $T_{2}$ is $\phi_{2}^{-1}\gamma_{q_{+D}}\phi_{2}=\phi_{2}^{-1}\cdot\gamma_{4}\cdot\phi_{2}=\gamma_{3}\gamma_{4}$.
\end{example*}
We summarise the information in Table \ref{tab:Monodromy_of_16} below,
the row headings are the $15$ singular hyperplanes and the column
headings are the $10$ nodes $q_{i}$. Each entry is an element $\gamma\in\Gamma$
and represents how changing the sign of the square root of that singular
hyperplane permutes the $16$ planes associated to $q_{i}$ (which
we know can be represented as an element of $\Gamma$). An empty box
stands for the identity element in $\Gamma$.

\begin{sidewaystable}
\begin{tabular}{|c|c|c|c|c|c|c|c|c|c|c|}
\cline{2-11} 
\multicolumn{1}{c|}{} & $q_{1}$ & $q_{2}$ & $q_{3}$ & $q_{4}$ & $q_{5}$ & $q_{6}$ & $q_{7}$ & $q_{8}$ & $q_{9}$ & $q_{10}$\tabularnewline
\hline 
$A$ &  &  &  &  & $\gamma_{3}$ & $\gamma_{3}$ & $\gamma_{4}$ & $\gamma_{4}$ & $\gamma_{3}\gamma_{4}$ & $\gamma_{3}\gamma_{4}$\tabularnewline
\hline 
$q_{+C}$ & $\gamma_{3}$ & $\gamma_{3}$ &  &  &  &  & $\gamma_{1}\gamma_{3}$ & $\gamma_{1}$ & $\gamma_{1}\gamma_{3}$ & $\gamma_{1}$\tabularnewline
\hline 
$-q_{-C}$ &  &  & $\gamma_{3}$ & $\gamma_{3}$ &  &  & $\gamma_{1}\gamma_{4}$ & $\gamma_{1}\gamma_{3}\gamma_{4}$ & $\gamma_{1}\gamma_{3}\gamma_{4}$ & $\gamma_{1}\gamma_{4}$\tabularnewline
\hline 
$q_{+D}$ & $\gamma_{4}$ &  & $\gamma_{4}$ &  & $\gamma_{2}\gamma_{4}$ & $\gamma_{2}$ &  &  & $\gamma_{2}\gamma_{4}$ & $\gamma_{2}$\tabularnewline
\hline 
$-q_{-D}$ &  & $\gamma_{4}$ &  & $\gamma_{4}$ & $\gamma_{2}\gamma_{3}$ & $\gamma_{2}\gamma_{3}\gamma_{4}$ &  &  & $\gamma_{2}\gamma_{3}\gamma_{4}$ & $\gamma_{2}\gamma_{3}$\tabularnewline
\hline 
$q_{+E}$ &  & $\gamma_{3}\gamma_{4}$ & $\gamma_{3}\gamma_{4}$ &  & $\gamma_{1}\gamma_{2}\gamma_{3}\gamma_{4}$ & $\gamma_{1}\gamma_{2}$ & $\gamma_{1}\gamma_{2}\gamma_{3}\gamma_{4}$ & $\gamma_{1}\gamma_{2}$ &  & \tabularnewline
\hline 
$-q_{-E}$ & $\gamma_{3}\gamma_{4}$ &  &  & $\gamma_{3}\gamma_{4}$ & $\gamma_{1}\gamma_{2}\gamma_{3}$ & $\gamma_{1}\gamma_{2}\gamma_{4}$ & $\gamma_{1}\gamma_{2}\gamma_{4}$ & $\gamma_{1}\gamma_{2}\gamma_{3}$ &  & \tabularnewline
\hline 
$p_{+0}$ & $\gamma_{1}\gamma_{2}$ & $\gamma_{2}$ & $\gamma_{1}$ &  & $\gamma_{1}$ &  & $\gamma_{2}$ &  & $\gamma_{1}\gamma_{2}$ & \tabularnewline
\hline 
$p_{-0}$ & $\gamma_{1}\gamma_{2}\gamma_{3}\gamma_{4}$ & $\gamma_{2}\gamma_{4}$ & $\gamma_{1}\gamma_{3}$ &  &  & $\gamma_{1}\gamma_{3}$ &  & $\gamma_{2}\gamma_{4}$ &  & $\gamma_{1}\gamma_{2}\gamma_{3}\gamma_{4}$\tabularnewline
\hline 
$p_{+1}$ & $\gamma_{2}\gamma_{3}\gamma_{4}$ & $\gamma_{1}\gamma_{2}\gamma_{4}$ &  & $\gamma_{1}\gamma_{3}$ & $\gamma_{1}\gamma_{3}$ &  &  & $\gamma_{2}\gamma_{3}\gamma_{4}$ &  & $\gamma_{1}\gamma_{2}\gamma_{4}$\tabularnewline
\hline 
$p_{-1}$ & $\gamma_{2}\gamma_{3}$ & $\gamma_{1}\gamma_{2}\gamma_{3}$ &  & $\gamma_{1}$ &  & $\gamma_{1}$ & $\gamma_{2}\gamma_{3}$ &  & $\gamma_{1}\gamma_{2}\gamma_{3}$ & \tabularnewline
\hline 
$p_{+2}$ & $\gamma_{1}\gamma_{3}\gamma_{4}$ &  & $\gamma_{1}\gamma_{2}\gamma_{3}$ & $\gamma_{2}\gamma_{4}$ &  & $\gamma_{1}\gamma_{3}\gamma_{4}$ & $\gamma_{2}\gamma_{4}$ &  &  & $\gamma_{1}\gamma_{2}\gamma_{3}$\tabularnewline
\hline 
$p_{-2}$ & $\gamma_{1}\gamma_{4}$ &  & $\gamma_{1}\gamma_{2}\gamma_{4}$ & $\gamma_{2}$ & $\gamma_{1}\gamma_{4}$ &  &  & $\gamma_{2}$ & $\gamma_{1}\gamma_{2}\gamma_{4}$ & \tabularnewline
\hline 
$p_{+3}$ &  & $\gamma_{1}\gamma_{4}$ & $\gamma_{2}\gamma_{3}$ & $\gamma_{1}\gamma_{2}\gamma_{3}\gamma_{4}$ &  & $\gamma_{1}\gamma_{4}$ &  & $\gamma_{2}\gamma_{3}$ & $\gamma_{1}\gamma_{2}\gamma_{3}\gamma_{4}$ & \tabularnewline
\hline 
$p_{-3}$ &  & $\gamma_{1}\gamma_{3}\gamma_{4}$ & $\gamma_{2}\gamma_{3}\gamma_{4}$ & $\gamma_{1}\gamma_{2}$ & $\gamma_{1}\gamma_{3}\gamma_{4}$ &  & $\gamma_{2}\gamma_{3}\gamma_{4}$ &  &  & $\gamma_{1}\gamma_{2}$\tabularnewline
\hline 
\end{tabular}

\caption{\label{tab:Monodromy_of_16}}
\end{sidewaystable}

\begin{thm}
The Monodromy group of the $160$ planes is $C_{2}^{9}$.\end{thm}
\begin{proof}
We use the information given in Table \ref{tab:Monodromy_of_16}.
The Monodromy group is a subgroup of $S_{160}$. It can be checked
that the subgroup of $S_{160}$ generated by $A,\dots,p_{-3}$ (once
they have been embedded in $S_{160}$) has order $2^{9}$. Now as
each of the generators $A,\dots,p_{-3}$ commutes with each other
and have order $2$, we know that every non-trivial element of the
Monodromy group have order $2$. Since the only group of order $2^{9}$
with every non-trivial elements being involutions is $C_{2}^{9}$,
the Monodromy group of the $160$ planes is $C_{2}^{9}$.\end{proof}
\begin{defn}
Each conic comes in a natural pair, i.e., each plane intersecting
the K3 surface gives two conics. We shall call two such conics \emph{conjugates}
of each other.
\end{defn}
Recall from Page \ref{pg:Q2} that given the K3 surface $X_{p}$ and
the plane $T:r_{0,1}x+r_{1,1}y+r_{2,1}z+r_{3,1}w=0$ associated to
the point $q_{1}$, then the two conics in $T\cap X_{p}$ are $Q_{1}+\sqrt{\frac{\mu_{1}}{\Delta}}Q'$
and $Q_{1}-\sqrt{\frac{\mu_{1}}{\Delta}}Q'$. Let $r_{\mu}=\sqrt{\frac{\mu}{\Delta}}\cdot a_{2}$,
then the conics are expressed as $a_{2}Q_{1}\pm r_{\mu}Q'$. Then,
using the same method as in the proof of Lemma \ref{lem:The_planes_q1},
we can express $r_{\mu}$and $a_{2}$ explicitly in terms of $A,B,C,D,E$
and find that:

{\small{
\begin{eqnarray*}
r_{\mu} & = & \frac{\sqrt{-q_{+C}q_{+D}q_{-E}}}{\sqrt{\Delta}}\bigg(b_{1}\sqrt{p_{-0}p_{+0}}+b_{2}\sqrt{p_{+1}p_{-1}}+b_{3}\sqrt{p_{+2}p_{-2}}+b_{4}\sqrt{p_{+0}p_{-0}p_{+1}p_{-1}p_{+2}p_{-2}})
\end{eqnarray*}
}}{\small \par}

\[
a_{2}=b_{5}\sqrt{p_{+1}p_{-1}p_{+2}p_{-2}}+b_{6}\sqrt{p_{+0}p_{-0}p_{+2}p_{-2}}+b_{7}\sqrt{p_{+0}p_{-0}p_{+1}p_{-1}}+b_{8}
\]
 where $b_{i}\in\mathbb{Z}[A,B,C,D,E]$. From the equation of the
conics $a_{2}Q_{1}\pm r_{\mu}Q'$ we can get the equations of the
other $30$ conics defined by the point $q_{1}$, using the action
of $\Gamma$. Hence we see that to calculate the Monodromy group of
the $32$ conics defined by the point $q_{1}$ we need to look at
loops going around the $15$ singular hyperplanes and the Segre cubic.
We first use the following lemma.
\begin{lem}
\label{lem:Lem}Let $\left|x\right|<\min\left\{ 1,\left|\frac{c}{2a}\right|,\left|\frac{c}{2b}\right|,\left|\frac{1}{a}\right|,\left|\frac{1}{b}\right|\right\} $,
then $x$ satisfies $\left|ax+bx^{2}\right|\leq\left|ax\right|+\left|bx^{2}\right|<\left|c\right|$.
In particular this implies that $c+ax+bx^{2}\neq0$.\end{lem}
\begin{proof}
This is a simple cases by cases proof:
\begin{casenv}
\item $\left|x\right|<1=\min\left\{ 1,\left|\frac{c}{2a}\right|,\left|\frac{c}{2b}\right|,\left|\frac{1}{a}\right|,\left|\frac{1}{b}\right|\right\} $.
Then $\left|ax\right|+\left|bx^{2}\right|<\left|a\right|+\left|b\right|$,
since $1\leq\left|\frac{c}{2a}\right|,\left|\frac{c}{2b}\right|$
we know that $a,b\leq\left|\frac{c}{2}\right|$. Hence $\left|a\right|+\left|b\right|\leq\left|c\right|$.
\item $\left|x\right|<\frac{1}{\left|a\right|}=\min\left\{ 1,\left|\frac{c}{2a}\right|,\left|\frac{c}{2b}\right|,\left|\frac{1}{a}\right|,\left|\frac{1}{b}\right|\right\} $.
Then $\left|ax\right|+\left|bx^{2}\right|<1+\left|\frac{b}{a^{2}}\right|\leq1+\left|\frac{b}{a}\right|$.
Since $\frac{1}{\left|a\right|}\leq\frac{1}{\left|b\right|}$ implies
$\left|\frac{b}{a}\right|\leq1$, and $\frac{1}{\left|a\right|}\leq\frac{\left|c\right|}{\left|2a\right|}$
implies $2\leq\left|c\right|$, then $1+\left|\frac{b}{a}\right|\leq\left|c\right|$.
\selectlanguage{english}%
\item $\left|x\right|<\frac{1}{\left|b\right|}=\min\left\{ 1,\left|\frac{c}{2a}\right|,\left|\frac{c}{2b}\right|,\left|\frac{1}{a}\right|,\left|\frac{1}{b}\right|\right\} $.
Then $\left|ax\right|+\left|bx^{2}\right|<\left|\frac{a}{b}\right|+\left|\frac{1}{b}\right|$.
As in case $2$, we see that $\left|\frac{a}{b}\right|\leq1$ and
$2\leq\left|c\right|$, hence $\left|\frac{a}{b}\right|+\left|\frac{1}{b}\right|\leq2\leq\left|c\right|$.
\item $\left|x\right|<\left|\frac{c}{2a}\right|=\min\left\{ 1,\left|\frac{c}{2a}\right|,\left|\frac{c}{2b}\right|,\left|\frac{1}{a}\right|,\left|\frac{1}{b}\right|\right\} $.
Then $\left|ax\right|+\left|bx^{2}\right|<\left|\frac{c}{2}\right|+\left|\frac{bc^{2}}{2a^{2}}\right|\leq\left|\frac{c}{2}\right|+\left|\frac{bc}{2a}\right|$.
As $\left|\frac{c}{2a}\right|\leq\left|\frac{c}{2b}\right|$ implies
$\left|\frac{b}{a}\right|\leq1$, we have that $\left|\frac{c}{2}\right|+\left|\frac{b}{a}\right|\left|\frac{c}{2}\right|\leq\left|c\right|$.
\item $\left|x\right|<\left|\frac{c}{2b}\right|=\min\left\{ 1,\left|\frac{c}{2a}\right|,\left|\frac{c}{2b}\right|,\left|\frac{1}{a}\right|,\left|\frac{1}{b}\right|\right\} $.
Then $\left|ax\right|+\left|bx^{2}\right|<\left|\frac{a}{b}\right|\left|\frac{c}{2}\right|+\left|\frac{c}{2b}\right|\left|\frac{c}{2}\right|\leq\left|\frac{c}{2}\right|+\left|\frac{c}{2}\right|\leq\left|c\right|$. 
\end{casenv}
\end{proof}
\begin{prop}
Given a non-singular K3 surface defined by the point $p=[A,B,C,D,E]$,
we can find a loop based at $p$ that changes the sign of $\sqrt{q_{+C}}$
and leaves the sign of $15$ square roots $\{\sqrt{-q_{-C}},\dots,\sqrt{p_{+3}},\sqrt{\Delta}\}$
unchanged.\end{prop}
\begin{proof}
We construct the same loop as in Theorem \ref{thm:Monodromy} with
a slight modification, this time we put the constraint that $\rho\in\rr_{>0}$
satisfies{\footnotesize{
\[
\rho<\min\big\{\left|B+2D+2E\right|,\left|-B+2D+2E\right|,\left|8A+B+2D-2E\right|,\left|8A-B+2D-2E\right|,1,
\]
\[
\frac{1}{\left|-4A^{2}+4DE\right|},\frac{1}{\left|4A\right|},\frac{\left|A^{2}B-4A(D^{2}+E^{2})+8ADE\right|}{2\left|-4A^{2}+4DE\right|},\frac{\left|A^{2}B-4A(D^{2}+E^{2})+8ADE\right|}{2\left|4A\right|}\bigg\}.
\]
}}The extra condition means that by Lemma \ref{lem:Lem} $\Delta=A^{2}B-4A(D^{2}+E^{2})+8ADE+(4CD-4A^{2})f(t)-4Af(t)^{2}\neq0$
during the second segment of the loop. At the same time, we see that
the sign of the square root cannot change. Finally, this extra condition
on $\rho$ does not effect the rest of the proof of Theorem \ref{thm:Monodromy}.
\end{proof}
We also prove that we can construct a loop based at $p$ which goes
around the Segre cubic but not the $15$ singular hyperplanes.
\begin{prop}
Given a non-singular K3 surface defined by the point $p=[A,B,C,D,E]$,
there exists a loop based at $p$ that changes the sign of $\sqrt{\Delta}$
and leaves the sign of the $15$ square roots $\{\sqrt{q_{+C}},\dots,\sqrt{p_{+3}}\}$. \end{prop}
\begin{proof}
First we claim that given such a K3 surface there exists $B'$ such
that the surface $[A,B',C,D,E]$ lies on the Segre cubic but not on
the 15 singular hyperplanes. The first part of the statement is easy
to see, just solve $\Delta=0$ in terms of $B'$, which, since $A\neq0$
as our surface is non-singular, has a solution as we are working over
$\mathbb{\overline{Q}}$ (in fact $B'$ is at worst in a degree $2$
extension of the field of definition of $A,B,C,D,E$). For the second
part, recall that if a point lies on the Segre cubic and one of the
15 singular hyperplanes, then it lies on a Segre plane, i.e., it must
lie on a further two singular hyperplanes. But note that any surface
lying on a Segre plane must lie on one singular hyperplane which is
defined with no $B$ (and hence $B'$) terms. Therefore if $[A,B',C,D,E]$
lied on such an singular hyperplane, then so would $[A,B,C,D,E]$,
contradicting our assumption that the surface is non-singular.

We will construct a loop in the same way as in proof of Theorem \ref{thm:Monodromy}.
Note that $B$ can be written uniquely as $B'+re^{i\phi}$ for some
$r\in\mathbb{R}_{>0}$ and $\phi\in[0,2\pi)$ . Define a loop $\widetilde{\gamma}(t)=[\widetilde{\gamma}_{A}(t),\widetilde{\gamma}_{B}(t),\widetilde{\gamma}_{C}(t),\widetilde{\gamma}_{D}(t),\widetilde{\gamma}_{E}(t)]$
($0\leq t\leq3$) as $\widetilde{\gamma}_{B}=iB'+f(t)$ and $\widetilde{\gamma}_{i}=i$
for all $ $$t$ and $i\in\{A,B,D,E\}$, where $f$ is composed of
the following $3$ segments:
\[
f(t)=\begin{cases}
(\rho t+r(1-t))e^{i\phi} & t\in[0,1]\\
\rho e^{i\phi+i(t-1)2\pi} & t\in[1,2]\\
(\rho(3-t)+r(t-2))e^{i\phi} & t\in[2,3]
\end{cases}
\]
and $\rho\in\rr_{>0}$ satisfies 
\begin{eqnarray*}
\rho & < & \min\{\left|4A\pm B'+2C+2D+2E\right|,\left|4A\pm B'+2C-2D-2E\right|,\\
 &  & \left|4A\pm B'-2C+2D-2E\right|,\left|4A\pm B'-2C-2D+2E\right|,\left|2B'\right|\}.
\end{eqnarray*}
Note that with the conditions on $\rho$ the path never loops around
the $15$ singular hyperplanes, hence as we have seen before, we do
not have a sign change from them. As for $\sqrt{\Delta}$, note that
the first and third segments leave it untouched, while for the second
segment 
\begin{eqnarray*}
\Delta & = & 16A^{3}-4A(C^{2}+D^{2}+E^{2})+4CDE+A(iB'+\rho e^{i\phi+i(t-1)e\pi})^{2}\\
 & = & -AB'^{2}+A(B'+\rho e^{i\phi+i(t-1)e\pi})^{2}\\
 & = & A\left(B'+\rho e^{i\phi+i(t-1)e\pi}+B'\right)\left(B'+\rho e^{i\phi+i(t-1)e\pi}-B'\right)\\
 & = & A\rho e^{i\phi+i(t-1)e\pi}\left(2B'+\rho e^{i\phi+i(t-1)e\pi}\right).
\end{eqnarray*}
Hence as in the previous proof, we find that as we loop around $\Delta=0$,
the sign of $\sqrt{\Delta}$ changes.

Note that the above only works under the assumption that $B'\neq0$.
In the case $B'=0$, we first need to find a path from our point $[A,B,C,D,E]$
to the point $[A+\epsilon,B,C,D,E]$ , where $\epsilon$ is small
enough that we do not go near any singular hyperplane or the Segre
cubic. In that case, we use the point $[A+\epsilon,B,C,D,E]$ as our
starting point.
\end{proof}
Hence we can use the explicit equations of the conics to find the
Monodromy group of the $32$ conics defined by the point $q_{1}$.
Then, as before, we use the group $\Omega$ acting on our set of points
$\{q_{i}\}$ and $15$ singular hyperplanes to find the Monodromy
group of the $32$ conics defined by each of the points $q_{i}$.
We summarized the information in Table \ref{tab:Monodromy_group},
where again the rows are the singular hyperplanes of the Segre cubic
we looped around and the columns are the points $q_{i}$. The entries
are either elements of $\Gamma$, or $-1$ which denotes the element
that conjugates conics, that is permutes conics defined on the same
plane.{\footnotesize{}}
\begin{sidewaystable}
\begin{tabular}{|c|c|c|c|c|c|c|c|c|c|c|}
\cline{2-11} 
\multicolumn{1}{c|}{} & $\beta_{1}$ & $\beta_{2}$ & $\beta_{3}$ & $\beta_{4}$ & $\beta_{5}$ & $\beta_{6}$ & $\beta_{7}$ & $\beta_{8}$ & $\beta_{9}$ & $\beta_{10}$\tabularnewline
\hline 
$\Delta$ & $-1$ & $-1$ & $-1$ & $-1$ & $-1$ & $-1$ & $-1$ & $-1$ & $-1$ & $-1$\tabularnewline
\hline 
$A$ &  &  &  &  & $-\gamma_{3}$ & $-\gamma_{3}$ & $-\gamma_{4}$ & $-\gamma_{4}$ & $-\gamma_{3}\gamma_{4}$ & $-\gamma_{3}\gamma_{4}$\tabularnewline
\hline 
$q_{+C}$ & $-\gamma_{3}$ & $-\gamma_{3}$ &  &  &  &  & $-\gamma_{1}\gamma_{3}$ & $-\gamma_{1}$ & $-\gamma_{1}\gamma_{3}$ & $-\gamma_{1}$\tabularnewline
\hline 
$-q_{-C}$ &  &  & $-\gamma_{3}$ & $-\gamma_{3}$ &  &  & $-\gamma_{1}\gamma_{4}$ & $-\gamma_{1}\gamma_{3}\gamma_{4}$ & $-\gamma_{1}\gamma_{3}\gamma_{4}$ & $-\gamma_{1}\gamma_{4}$\tabularnewline
\hline 
$q_{+D}$ & $-\gamma_{4}$ &  & $-\gamma_{4}$ &  & $-\gamma_{2}\gamma_{4}$ & $-\gamma_{2}$ &  &  & $-\gamma_{2}\gamma_{4}$ & $-\gamma_{2}$\tabularnewline
\hline 
$-q_{-D}$ &  & $-\gamma_{4}$ &  & $-\gamma_{4}$ & $-\gamma_{2}\gamma_{3}$ & $-\gamma_{2}\gamma_{3}\gamma_{4}$ &  &  & $-\gamma_{2}\gamma_{3}\gamma_{4}$ & $-\gamma_{2}\gamma_{3}$\tabularnewline
\hline 
$q_{+E}$ &  & $-\gamma_{3}\gamma_{4}$ & $-\gamma_{3}\gamma_{4}$ &  & $-\gamma_{1}\gamma_{2}\gamma_{3}\gamma_{4}$ & $-\gamma_{1}\gamma_{2}$ & $-\gamma_{1}\gamma_{2}\gamma_{3}\gamma_{4}$ & $-\gamma_{1}\gamma_{2}$ &  & \tabularnewline
\hline 
$-q_{-E}$ & $-\gamma_{3}\gamma_{4}$ &  &  & $-\gamma_{3}\gamma_{4}$ & $-\gamma_{1}\gamma_{2}\gamma_{3}$ & $-\gamma_{1}\gamma_{2}\gamma_{4}$ & $-\gamma_{1}\gamma_{2}\gamma_{4}$ & $-\gamma_{1}\gamma_{2}\gamma_{3}$ &  & \tabularnewline
\hline 
$p_{+0}$ & $-\gamma_{1}\gamma_{2}$ & $-\gamma_{2}$ & $-\gamma_{1}$ &  & $-\gamma_{1}$ &  & $-\gamma_{2}$ &  & $-\gamma_{1}\gamma_{2}$ & \tabularnewline
\hline 
$p_{-0}$ & $-\gamma_{1}\gamma_{2}\gamma_{3}\gamma_{4}$ & $-\gamma_{2}\gamma_{4}$ & $-\gamma_{1}\gamma_{3}$ &  &  & $-\gamma_{1}\gamma_{3}$ &  & $-\gamma_{2}\gamma_{4}$ &  & $-\gamma_{1}\gamma_{2}\gamma_{3}\gamma_{4}$\tabularnewline
\hline 
$p_{+1}$ & $-\gamma_{2}\gamma_{3}\gamma_{4}$ & $-\gamma_{1}\gamma_{2}\gamma_{4}$ &  & $-\gamma_{1}\gamma_{3}$ & $-\gamma_{1}\gamma_{3}$ &  &  & $-\gamma_{2}\gamma_{3}\gamma_{4}$ &  & $-\gamma_{1}\gamma_{2}\gamma_{4}$\tabularnewline
\hline 
$p_{-1}$ & $-\gamma_{2}\gamma_{3}$ & $-\gamma_{1}\gamma_{2}\gamma_{3}$ &  & $-\gamma_{1}$ &  & $-\gamma_{1}$ & $-\gamma_{2}\gamma_{3}$ &  & $-\gamma_{1}\gamma_{2}\gamma_{3}$ & \tabularnewline
\hline 
$p_{+2}$ & $-\gamma_{1}\gamma_{3}\gamma_{4}$ &  & $-\gamma_{1}\gamma_{2}\gamma_{3}$ & $-\gamma_{2}\gamma_{4}$ &  & $-\gamma_{1}\gamma_{3}\gamma_{4}$ & $-\gamma_{2}\gamma_{4}$ &  &  & $-\gamma_{1}\gamma_{2}\gamma_{3}$\tabularnewline
\hline 
$p_{-2}$ & $-\gamma_{1}\gamma_{4}$ &  & $-\gamma_{1}\gamma_{2}\gamma_{4}$ & $-\gamma_{2}$ & $-\gamma_{1}\gamma_{4}$ &  &  & $-\gamma_{2}$ & $-\gamma_{1}\gamma_{2}\gamma_{4}$ & \tabularnewline
\hline 
$p_{+3}$ &  & $-\gamma_{1}\gamma_{4}$ & $-\gamma_{2}\gamma_{3}$ & $-\gamma_{1}\gamma_{2}\gamma_{3}\gamma_{4}$ &  & $-\gamma_{1}\gamma_{4}$ &  & $-\gamma_{2}\gamma_{3}$ & $-\gamma_{1}\gamma_{2}\gamma_{3}\gamma_{4}$ & \tabularnewline
\hline 
$p_{-3}$ &  & $-\gamma_{1}\gamma_{3}\gamma_{4}$ & $-\gamma_{2}\gamma_{3}\gamma_{4}$ & $-\gamma_{1}\gamma_{2}$ & $-\gamma_{1}\gamma_{3}\gamma_{4}$ &  & $-\gamma_{2}\gamma_{3}\gamma_{4}$ &  &  & $-\gamma_{1}\gamma_{2}$\tabularnewline
\hline 
\end{tabular}

{\footnotesize{\caption{\label{tab:Monodromy_group}}
}}
\end{sidewaystable}
{\footnotesize \par}
\begin{thm}
The Monodromy group of the $320$ conics is $C_{2}^{10}$.\end{thm}
\begin{proof}
We use the information given in Table \ref{tab:Monodromy_group}.
The Monodromy group is a subgroup of $S_{320}$. It can be checked
that the subgroup of $S_{320}$ generated by $\Delta,\dots,p_{-3}$
(once they have been embedded in $S_{320}$) has order $2^{10}$.
Now as each of the generators $\Delta,\dots,p_{-3}$ commutes with
each other and have order $2$, we know that every non-trivial element
of the Monodromy group have order $2$. Since the only group of order
$2^{10}$ with every non-trivial elements being involutions is $C_{2}^{10}$,
the Monodromy group of the 320 conics is $C_{2}^{10}$.\end{proof}
\begin{cor}
The moduli space of pairs $(X,C)$, where $X$ is an Heisenberg-invariant
quartic K3 surfaces and $C$ a conic on $X$, has 10 irreducible components.\end{cor}
\begin{proof}
Let $Z$ be the moduli space of pairs $(X,C)$ with $X$ a surface
in $\mathcal{X}$ and $C$ a conic lying on $X$. We showed that the
Monodromy group of $\pi$, $\pi:Z\to\pp_{[A,B,C,D,E]}^{4}$, breaks
the $320$ conics on $X$ in $10$ orbits, each orbit having size
$32$. Since calculating the Monodromy group involves lifting a path
in $\pp_{[A,B,C,D,E]}^{4}$ to a path in $Z$, any two elements in
the same orbit represent two connected elements in $Z$. Finally,
since the paths avoided where $\pi$ was not smooth, the $10$ orbits
correspond to $10$ smooth connected components of $Z$, i.e. $10$
irreducible components. 
\end{proof}
\bibliographystyle{amsalpha}
\bibliography{On the Monodromy and Galois Group of Conics}

\end{document}